\documentclass[11pt,letterpaper]{amsart}
\usepackage{graphicx}
\usepackage{amsmath, amssymb, amsthm, tikz, mathrsfs,verbatim,tensor}
\usepackage[left=1in,top=1in,right=1in]{geometry}
\usepackage{setspace}
\usepackage{enumerate}
\usepackage{amsfonts}
\usepackage{latexsym}
\usepackage{graphics}

\usepackage{bbm}

\newtheorem{theorem}{Theorem}[section]
\newtheorem{lemma}[theorem]{Lemma}
\newtheorem{corollary}[theorem]{Corollary}

\newtheorem{definition}[theorem]{Definition}

\theoremstyle{definition}
\newtheorem*{remark}{Remark}
\theoremstyle{plain}

\def\CS{\mathcal S}

\newcommand{\beqlbl}{\begin{equation}}
\newcommand{\eeqlbl}{\end{equation}}

\newcommand{\dyck}[1]{\text{Dyck}^{#1}}

\newcommand{\dn}{\dyck{2n}}
\newcommand{\sn}{S_n}

\newcommand{\thingone}{\tau}
\newcommand{\thingtwo}{\rho}

\newcommand{\R}{\ensuremath{\mathbb{R}} }

\newcommand{\N}{\ensuremath{\mathbb{N}} }

\newcommand{\Z}{\ensuremath{\mathbb{Z}} }
\newcommand{\D}{\ensuremath{\mathcal{D}} }

\newcommand{\A}{\mathcal{A}}

\renewcommand{\P}{\mathbb{P}}

\newcommand{\var}{\textrm{Var}}
\newcommand{\E}{\mathbb{E}}
\newcommand{\ind}{\mathbf{1}}

\newcommand{\exc}{\mathbbm{e}}

\usepackage{hyperref}
\usepackage{tikz}

\tikzstyle{vertex}=[circle, draw, inner sep=0pt, minimum size=6pt]
\tikzstyle{Vertex}=[circle, draw, inner sep=0pt, minimum size=14pt]
\tikzstyle{Vertexc}=[circle, draw, inner sep=0pt, minimum size=14pt, fill=blue!30]
\tikzstyle{vertexc}=[circle, draw, inner sep=0pt, minimum size=6pt, fill=red!40]
\tikzstyle{vertexcg}=[circle, draw, inner sep=0pt, minimum size=6pt, fill=green!70!black]

\newenvironment{pfofthm}[1]
{\par\vskip2\parsep\noindent{\em Proof of Theorem\ #1. }}{{\hfill
$\blacksquare$}
\par\vskip2\parsep}

\newenvironment{pfofcor}[1]
{\par\vskip2\parsep\noindent{\em Proof of Corollary\ #1. }}{{\hfill
$\blacksquare$}
\par\vskip2\parsep}

\newenvironment{pfoflem}[1]
{\par\vskip2\parsep\noindent{\em Proof of Lemma\ #1. }}{{\hfill
$\blacksquare$}
\par\vskip2\parsep}

\newcommand{\beq}{\begin{equation*}}
\newcommand{\eeq}{\end{equation*}}
\newcommand{\ba}{\begin{align*}}
\newcommand{\ea}{\end{align*}}

\newcommand{\field}[1]{\mathbb{#1}}

\newcommand{\matbegin}[1]{\left (  \begin{array} {#1} }
\newcommand{\matend}{ \end{array} \right ) } 
\newcommand{\prob}{\field{P}}
\newcommand{\expect}{\field{E}}

\newcommand{\indc}{\mathbf{1}}

\newcommand{\fp}{\theta}

\newcommand{\Iint}[1]{I_{#1}^{int}}
\newcommand{\Iout}[1]{I_{#1}^{out}}

\newcommand{\cE}{\mathcal{E}}

\newcommand{\shortexc}{\text{SE}}

\title[Fixed points of pattern-avoiding permutations]{Pattern-avoiding permutations and Brownian excursion, Part II: Fixed points}

\author[Christopher~Hoffman]{ \ Christopher~Hoffman$^\star\ddagger$}
\author[Douglas~Rizzolo]{ \ Douglas~Rizzolo$^{\star\dagger}$}
\author[Erik~Slivken]{ \ Erik~Slivken$^{\triangle \circ}$}

\thanks{\thinspace ${\hspace{-.45ex}}^\star$
Department of Mathematics,
University of Washington, Seattle, WA, 98195.
\hfill \\
\thinspace ${\hspace{-.45ex}}^\triangle$
Department of Mathematics,
University of California, Davis
Davis CA, 95616.
\hfill \\
${\hspace{-.45ex}}^\dagger$ 
Supported by NSF grant DMS-1204840
\hfil \\
${\hspace{-.45ex}}^\ddagger$ 
Supported by NSF grant DMS-1308645 and NSA grant H98230-13-1-0827
\hfil \\
${\hspace{-.45ex}}^\circ$ 
Supported by NSF RTG grant 0838212
\hfil \\
Email:
\hskip.02cm
\texttt{\{hoffman,drizzolo\}@math.washington.edu; erikslivken@math.ucdavis.edu}
}

\date{\today}

\begin{document}

\maketitle

\begin{abstract}
Permutations that avoid given patterns are among the most classical objects in combinatorics and have strong connections to many fields of mathematics, computer science and biology. In this paper we study fixed points of both 123- and 231-avoiding permutations.  We find an exact description for a scaling limit of the empirical distribution of fixed points in term of Brownian excursion.  This builds on the connections between pattern-avoiding permutations and Brownian excursion developed in \cite{part1} and strengthens the recent results of Elizalde \cite{elizalde2} and Miner and Pak \cite{mp} on fixed points of pattern-avoiding permutations.
\end{abstract}

\section{Introduction}

\noindent

\smallskip

In this paper we study the asymptotic behavior of the fixed points of pattern-avoiding permutations.  The study of random pattern-avoiding permutations has drawn considerable attention in the recent literature.  A body of work has developed around studying geometric propoerties of the graph of the permutation.  A suprising result was that Brownian excursion began to appear, in various guises, in descriptions of the limiting objects, see for example the recent work of Janson \cite{Ja14}, Madras and Liu \cite{madras2010random}, Madras and Pehlivan \cite{ML} and Miner and Pak \cite{mp}.  In Part I of this series \cite{part1}, we gave a strong pathwise connection bewteen the graph of a pattern avoding permutation and Brownian excursion that explains the large scale behavior of the graph.  This result, however, does not immediately yield information about local properties, such as fixed points of the permutation.  The fixed points of pattern-avoiding permutations have drawn special attention in the literature, see for example \cite{elizalde1, elizalde2004statistics, elizalde2, elizalde3, mp}.  In this paper we show that the fixed points of pattern-avoiding permutations are related to Brownian excursion.  This result is suprising because, although Brownian excusion is related to the bulk behavior of the graph of a pattern-avoiding permutation \cite{part1}, the property of being a fixed point is a very local property. 

Our main result is the following theorem, which as far as we know is the first to give a connection between the asymptotic distribution of fixed points of pattern-avoiding permutations and Brownian excursion.  Recall that if $\pi \in \CS_k$ and $\tau\in \CS_n$, we say that $\tau$ contains the pattern $\pi$ if there exist $i_1<i_2<\cdots <i_k$ such that for all $1\leq r<s\leq k$ we have $\pi(r)<\pi(s)$ if and only if $\tau(i_r)<\tau(i_s)$.  We say $\tau$ avoids $\pi$, or is $\pi$-avoiding, if $\tau$ does not contain $\pi$.

\pagebreak
\begin{theorem}\label{intro main}
Let $(\exc_t,0\leq t\leq 1)$ be standard Brownian exursion and let $\sigma_n$ and $\rho_n$ be respectively a uniformly random $\mathbf{231}$-avoiding permutation of $[n]$ and a uniformly random $\mathbf{123}$-avoiding permutation of $[n]$.  Then
\begin{enumerate}
\item \[ \lim_{n\to\infty} \frac{1}{n^{1/4}} \sum_{i=1}^n \delta_{i/n}\ind_{\{\sigma_n(i)=i\}} =_d \frac{1}{2^{7/4}\pi^{1/2}} \exc_t^{-3/2} dt,\]
where the convergence is with respect to weak convergence of finite measures on $[0,1]$. 
\item\label{intro main 2} Let $A$ and $B$ be independent $Bernoulli(1/4)$ random variables, also jointly independent of $(\exc_t,0\leq t\leq 1)$.  Then
\[ \lim_{n\to\infty} \sum_{i=1}^n \delta_{\frac{i-\frac{n}{2}}{\sqrt{2n}}} \ind_{\{\rho_n(i)=i\}} =_d A\delta_{-\exc(1/2)/2}+ B\delta_{\exc(1/2)/2},\]
where the convergence is with respect to weak convergence of finite measures on $\R$.
\end{enumerate}
\end{theorem}

This result builds on the large scale connection between pattern-avoiding permutations and Brownian excursion developed in \cite{part1}, where it was used to show that the bulk of a pattern-avoiding permutation can be asymptotically described by Brownian excursion.  Part (b) of Theorem \ref{intro main} has a nice interpretation.  A $\mathbf{123}$-avoiding permutation can have at most two fixed points, one above $n/2$ and one below it.  Part (b) of the theorem says that, asymptotically, these fixed points occur independently.  Moreover, conditionally given that both fixed points exist they are reflections of eachother across $n/2$ and the fluctuation of their distance from $n/2$ is given by the midpoint of Brownian excursion.  We emphisize that the limiting measure has the additional randomness of $(A,B)$ that is not part of the Brownian excursion.  In the proof we will see that this is a consequence of the fact that having $\sigma(i)=i$ is in a sense a local property of the permutation.  In Part (a) of  Theorem \ref{intro main}, such extra randomness is present at the discrete level, but does not appear in the limit for reasons related to the Law of Large Numbers.  

The appearance of Brownian excursion in Theorem \ref{intro main} will be explained by particular bijections between pattern-avoiding permutations and Dyck paths.  The bijection we use for $\mathbf{231}$-avoiding permutations was first used in Part I \cite{part1} and is better suited to extracting probabilistic information than more classical bijections, while the bijection we use for $\mathbf{123}$-avoiding permutations is a classical bijection from \cite{BJS}. 

The fixed points of random permutations have been well studied in both probability and combinatorics.  We will not survey the field here, but for the sake of comparrison we state the classical result of Montmort and Bernoulli  \cite{de1714essai} on the distribution of the number of fixed points in a uniformly chosen random permutation in language similar to ours.

\begin{theorem}[\cite{de1714essai}]\label{classical}
Let $\pi_n$ be a uniformly random permutation of $[n]$ and let $N$ be a Poisson random measure on $[0,1]$ with intensity equal to Lebesgue measure.  Then
\[ \lim_{n\to\infty} \sum_{i=1}^n \delta_{i/n}\ind_{\{\pi_n(i)=i\}} =_d N,\]
where the convergence is with respect to weak convergence of finite measures on $[0,1]$. 
\end{theorem}

Montmort \cite{de1714essai} shows that the number of fixed points converges to a Poisson random variable, but the extension to convergence of the empirical distribution of fixed points to a Poisson random measure is straight-forward, see e.g \cite[Theorem 11]{CDM05} for a strong version of this result based on Stein's method. Comparing Theorems \ref{intro main} and \ref{classical}, we see that $\mathbf{231}$-avoiding permutations have many more fixed points that uniformly random permutations and these fixed points are more likely to appear near $1$ or $n$, while $\mathbf{123}$-avoiding permutations have fewer fixed points than uniformly random permutations and they are more closely concentrated around $n/2$.  

The previous strongest results on the fixed points of pattern-avoiding permutations were established in \cite{elizalde2004statistics, mp}, which we summarize in the following theorem.

\begin{theorem}\label{intro summary}
For a permutation $\pi$, let $\mathrm{fp}(\pi)$ be the number of fixed points of $\pi$.
\begin{enumerate}
\item (Theorem 6.4 \cite{mp}) Let $\sigma_n$ be a uniformly random $\mathbf{231}$-avoiding permutation of $[n]$. Then
\[\lim_{n\to\infty} n^{-1/4} \E(\mathrm{fp}(\sigma_n)) = \frac{\mathrm{Gamma}\left(\frac{1}{4}\right)}{2\sqrt{\pi}}.\]
\item (Proposition 5.3 \cite{elizalde2004statistics}, Theorems 6.3 \cite{mp}) Let $\rho_n$ be a uniformly random $\mathbf{123}$-avoiding permuation of $[n]$.  Then
\[\lim_{n\to\infty} \E(\mathrm{fp}(\rho_n)) = \frac{1}{2}\]
and for every $\epsilon>0$ 
\[\lim_{n\to\infty} \P\left(\rho_n(i)=i \textrm{ for some } i \textrm{ such that } \left| i - \frac{n}{2}\right| > \epsilon n\right) =0.\]
\end{enumerate}
\end{theorem}
We remark that there is a small mistake in \cite[Theorem 6.4]{mp}, where the limit in Part (a) is given as $2\mathrm{Gamma}\left(\frac{1}{4}\right)/\sqrt{\pi}$, but it is easily seen from the proof of \cite[Theorem 6.4]{mp} that the value we give here is correct.  From this we see that our results in Theorem \ref{intro main} are the first to give detailed information about the asymptotic distribution of fixed points of pattern-avoiding permutations.

\subsection{Connections with invariance principles}
In this section we give more detail on the relationship between Theorem \ref{intro main} and the results of Part I \cite{part1}.  As in \cite{part1}, our results here are derived from bijections between Dyck paths and pattern-avoiding permutations.
Throughout the paper we use the following definition of a Dyck path.
\begin{definition}   A {\bf Dyck path} $\gamma$ is a sequence  $\{\gamma(x)\}_{x=0}^{2n}$ that satisfy the following conditions:

\begin{itemize}
\item $\gamma(0)=\gamma(2n)=0$
\item $\gamma(x) \geq 0$ for all $x \in \{0,1,\dots, 2n\}$ and
\item $|\gamma(x+1)-\gamma(x)| = 1$ for all $x \in \{0,1,\dots, 2n-1\}.$
\end{itemize}
\end{definition}

We often want to consider the function generated by a Dyck path through linear interpolation. Throughout this paper we often use the same notation to denote a sequence and the continuous function generated by extending it through linear interpolation.

Brownian excursion is the process 
 $(\mathbbm{e}_t)_{0 \leq t \leq 1}$  which is Brownian motion conditioned to be 0 at 0 and 1 and positive in the interior \cite{morters2010brownian}.
It is well known that the scaling limit of Dyck paths are Brownian excursion \cite{Ka76}  and that Dyck Paths of length 2n are in bijection with 
\textbf{321}-avoiding and \textbf{231}-avoiding permutations \cite{Knu, M}.

\subsection*{Fixed points of \textbf{123}-avoiding permutations}

The convergence developed in \cite{part1} is strongly suggestive of the general form of the limit distribution of fixed points of $\mathbf{123}$-avoiding permutations.  To see this, we rephrase \cite[Theorem 1.2]{part1} in terms of $\textbf{123}$-avoiding permutations.

\begin{theorem}{\cite[Theorem 1.2]{part1}} \label{123 invariance}
Let $\rho_n$ be a uniformly random $\mathbf{123}$-avoiding permutation of $[n]$.  Then there exist a (random) partition of $[n]$ into $S^+$ and $S^-$ such that if $F_n^{\pm}:[0,n+1]\to \R$ is the linear interpolation of the points
\[  \left\{ \left(i,\rho_n(i)\right) : i\in S^{\pm}\right\} \cup \{(0,n)\} \cup \{(n+1,0)\}\]
Then
\[ \left( \frac{F_n^+(nt) - n(1-t)}{\sqrt{2n}} , \frac{F_n^-(nt) - n(1-t)}{\sqrt{2n}} \right)_{t\in [0,1]} \Rightarrow (\exc_t,-\exc_t)_{t\in [0,1]}.\]
\end{theorem}

Theorem \ref{123 invariance} shows that $\rho_n(i) = n-i + O(\sqrt{n})$, where the fluctuation is described by Brownian excursion.  Consequently, if $\rho_n(i)=i$ then $i =  n-i + O(\sqrt{n})$, so that $i = (n/2)+ O(\sqrt{n})$.  Thus if $\rho_n$ has any fixed points then they then they are within $O(\sqrt{n})$ of $n/2$.  This already gives an improvement over Part (b) of Theorem \ref{intro summary} in terms of the location of the fixed points.  To establish Part (b) of Theorem \ref{intro main}, we must carefully examine the local structure of $\rho_n$ and this is what leads to the independent Bernoulli random variables appearing in the theorem.  This will be done using the bijection with Dyck paths introduced in Section \ref{offer}.  

Although Theorem \ref{123 invariance} strongly suggests the general form of Theorem \ref{intro main} Part (b), our proof of Theorem \ref{intro main} does not depend on Theorem \ref{123 invariance}.  We also remark that \cite[Theorem 1.2]{part1} is stated for $\mathbf{321}$-avoiding permutations, however, the fixed points of {\bf 321}-avoiding permutations are concentrated near 1 and $n$ and are not well-described by Brownian excursion.

\begin{figure}
\centering
\begin{tikzpicture}
\draw[ultra thin](0, 0) --
(1/150, 1/13) --
(1/75, 0) --
(1/50, 1/13) --
(2/75, 0) --
(1/30, 1/13) --
(1/25, 0) --
(7/150, 1/13) --
(4/75, 2/13) --
(3/50, 3/13) --
(1/15, 2/13) --
(11/150, 3/13) --
(2/25, 2/13) --
(13/150, 1/13) --
(7/75, 2/13) --
(1/10, 3/13) --
(8/75, 4/13) --
(17/150, 3/13) --
(3/25, 4/13) --
(19/150, 5/13) --
(2/15, 4/13) --
(7/50, 5/13) --
(11/75, 4/13) --
(23/150, 3/13) --
(4/25, 4/13) --
(1/6, 5/13) --
(13/75, 4/13) --
(9/50, 5/13) --
(14/75, 6/13) --
(29/150, 7/13) --
(1/5, 8/13) --
(31/150, 9/13) --
(16/75, 8/13) --
(11/50, 7/13) --
(17/75, 8/13) --
(7/30, 7/13) --
(6/25, 6/13) --
(37/150, 7/13) --
(19/75, 6/13) --
(13/50, 7/13) --
(4/15, 8/13) --
(41/150, 9/13) --
(7/25, 10/13) --
(43/150, 9/13) --
(22/75, 10/13) --
(3/10, 9/13) --
(23/75, 8/13) --
(47/150, 9/13) --
(8/25, 10/13) --
(49/150, 9/13) --
(1/3, 10/13) --
(17/50, 11/13) --
(26/75, 10/13) --
(53/150, 9/13) --
(9/25, 8/13) --
(11/30, 7/13) --
(28/75, 6/13) --
(19/50, 5/13) --
(29/75, 6/13) --
(59/150, 7/13) --
(2/5, 8/13) --
(61/150, 7/13) --
(31/75, 8/13) --
(21/50, 7/13) --
(32/75, 8/13) --
(13/30, 7/13) --
(11/25, 6/13) --
(67/150, 7/13) --
(34/75, 8/13) --
(23/50, 7/13) --
(7/15, 8/13) --
(71/150, 9/13) --
(12/25, 8/13) --
(73/150, 9/13) --
(37/75, 8/13) --
(1/2, 9/13) --
(38/75, 8/13) --
(77/150, 9/13) --
(13/25, 10/13) --
(79/150, 9/13) --
(8/15, 8/13) --
(27/50, 9/13) --
(41/75, 10/13) --
(83/150, 11/13) --
(14/25, 10/13) --
(17/30, 11/13) --
(43/75, 12/13) --
(29/50, 11/13) --
(44/75, 10/13) --
(89/150, 9/13) --
(3/5, 10/13) --
(91/150, 9/13) --
(46/75, 10/13) --
(31/50, 9/13) --
(47/75, 10/13) --
(19/30, 11/13) --
(16/25, 10/13) --
(97/150, 9/13) --
(49/75, 8/13) --
(33/50, 9/13) --
(2/3, 10/13) --
(101/150, 9/13) --
(17/25, 8/13) --
(103/150, 9/13) --
(52/75, 10/13) --
(7/10, 11/13) --
(53/75, 12/13) --
(107/150, 11/13) --
(18/25, 10/13) --
(109/150, 9/13) --
(11/15, 10/13) --
(37/50, 9/13) --
(56/75, 8/13) --
(113/150, 9/13) --
(19/25, 8/13) --
(23/30, 7/13) --
(58/75, 6/13) --
(39/50, 5/13) --
(59/75, 6/13) --
(119/150, 5/13) --
(4/5, 4/13) --
(121/150, 5/13) --
(61/75, 4/13) --
(41/50, 5/13) --
(62/75, 4/13) --
(5/6, 3/13) --
(21/25, 2/13) --
(127/150, 3/13) --
(64/75, 2/13) --
(43/50, 1/13) --
(13/15, 2/13) --
(131/150, 3/13) --
(22/25, 2/13) --
(133/150, 3/13) --
(67/75, 2/13) --
(9/10, 3/13) --
(68/75, 4/13) --
(137/150, 5/13) --
(23/25, 4/13) --
(139/150, 3/13) --
(14/15, 2/13) --
(47/50, 3/13) --
(71/75, 4/13) --
(143/150, 5/13) --
(24/25, 6/13) --
(29/30, 5/13) --
(73/75, 4/13) --
(49/50, 5/13) --
(74/75, 4/13) --
(149/150, 5/13) --
(1, 6/13) --
(151/150, 7/13) --
(76/75, 6/13) --
(51/50, 5/13) --
(77/75, 4/13) --
(31/30, 5/13) --
(26/25, 4/13) --
(157/150, 5/13) --
(79/75, 4/13) --
(53/50, 3/13) --
(16/15, 2/13) --
(161/150, 1/13) --
(27/25, 2/13) --
(163/150, 3/13) --
(82/75, 4/13) --
(11/10, 5/13) --
(83/75, 6/13) --
(167/150, 7/13) --
(28/25, 8/13) --
(169/150, 9/13) --
(17/15, 10/13) --
(57/50, 9/13) --
(86/75, 10/13) --
(173/150, 11/13) --
(29/25, 12/13) --
(7/6, 11/13) --
(88/75, 12/13) --
(59/50, 1) --
(89/75, 12/13) --
(179/150, 11/13) --
(6/5, 12/13) --
(181/150, 11/13) --
(91/75, 12/13) --
(61/50, 11/13) --
(92/75, 12/13) --
(37/30, 11/13) --
(31/25, 10/13) --
(187/150, 11/13) --
(94/75, 10/13) --
(63/50, 9/13) --
(19/15, 8/13) --
(191/150, 9/13) --
(32/25, 8/13) --
(193/150, 7/13) --
(97/75, 6/13) --
(13/10, 7/13) --
(98/75, 6/13) --
(197/150, 7/13) --
(33/25, 8/13) --
(199/150, 7/13) --
(4/3, 8/13) --
(67/50, 9/13) --
(101/75, 10/13) --
(203/150, 9/13) --
(34/25, 10/13) --
(41/30, 11/13) --
(103/75, 10/13) --
(69/50, 9/13) --
(104/75, 10/13) --
(209/150, 9/13) --
(7/5, 10/13) --
(211/150, 9/13) --
(106/75, 8/13) --
(71/50, 9/13) --
(107/75, 10/13) --
(43/30, 11/13) --
(36/25, 10/13) --
(217/150, 9/13) --
(109/75, 10/13) --
(73/50, 9/13) --
(22/15, 10/13) --
(221/150, 9/13) --
(37/25, 8/13) --
(223/150, 9/13) --
(112/75, 8/13) --
(3/2, 9/13) --
(113/75, 10/13) --
(227/150, 9/13) --
(38/25, 10/13) --
(229/150, 9/13) --
(23/15, 10/13) --
(77/50, 11/13) --
(116/75, 10/13) --
(233/150, 11/13) --
(39/25, 10/13) --
(47/30, 9/13) --
(118/75, 10/13) --
(79/50, 11/13) --
(119/75, 12/13) --
(239/150, 1) --
(8/5, 12/13) --
(241/150, 11/13) --
(121/75, 10/13) --
(81/50, 11/13) --
(122/75, 12/13) --
(49/30, 11/13) --
(41/25, 12/13) --
(247/150, 1) --
(124/75, 14/13) --
(83/50, 15/13) --
(5/3, 14/13) --
(251/150, 15/13) --
(42/25, 16/13) --
(253/150, 17/13) --
(127/75, 16/13) --
(17/10, 15/13) --
(128/75, 16/13) --
(257/150, 15/13) --
(43/25, 16/13) --
(259/150, 17/13) --
(26/15, 16/13) --
(87/50, 17/13) --
(131/75, 16/13) --
(263/150, 15/13) --
(44/25, 14/13) --
(53/30, 15/13) --
(133/75, 14/13) --
(89/50, 15/13) --
(134/75, 14/13) --
(269/150, 1) --
(9/5, 14/13) --
(271/150, 15/13) --
(136/75, 14/13) --
(91/50, 1) --
(137/75, 14/13) --
(11/6, 1) --
(46/25, 14/13) --
(277/150, 15/13) --
(139/75, 16/13) --
(93/50, 15/13) --
(28/15, 14/13) --
(281/150, 15/13) --
(47/25, 14/13) --
(283/150, 15/13) --
(142/75, 16/13) --
(19/10, 15/13) --
(143/75, 16/13) --
(287/150, 15/13) --
(48/25, 16/13) --
(289/150, 17/13) --
(29/15, 18/13) --
(97/50, 17/13) --
(146/75, 18/13) --
(293/150, 19/13) --
(49/25, 20/13) --
(59/30, 19/13) --
(148/75, 18/13) --
(99/50, 19/13) --
(149/75, 18/13) --
(299/150, 19/13) --
(2, 18/13) --
(301/150, 17/13) --
(151/75, 16/13) --
(101/50, 15/13) --
(152/75, 16/13) --
(61/30, 17/13) --
(51/25, 16/13) --
(307/150, 15/13) --
(154/75, 16/13) --
(103/50, 15/13) --
(31/15, 16/13) --
(311/150, 17/13) --
(52/25, 16/13) --
(313/150, 17/13) --
(157/75, 18/13) --
(21/10, 19/13) --
(158/75, 18/13) --
(317/150, 19/13) --
(53/25, 20/13) --
(319/150, 21/13) --
(32/15, 20/13) --
(107/50, 19/13) --
(161/75, 18/13) --
(323/150, 17/13) --
(54/25, 16/13) --
(13/6, 15/13) --
(163/75, 14/13) --
(109/50, 1) --
(164/75, 12/13) --
(329/150, 1) --
(11/5, 12/13) --
(331/150, 11/13) --
(166/75, 10/13) --
(111/50, 9/13) --
(167/75, 8/13) --
(67/30, 9/13) --
(56/25, 8/13) --
(337/150, 9/13) --
(169/75, 10/13) --
(113/50, 11/13) --
(34/15, 10/13) --
(341/150, 11/13) --
(57/25, 12/13) --
(343/150, 11/13) --
(172/75, 12/13) --
(23/10, 1) --
(173/75, 12/13) --
(347/150, 11/13) --
(58/25, 10/13) --
(349/150, 11/13) --
(7/3, 10/13) --
(117/50, 11/13) --
(176/75, 10/13) --
(353/150, 9/13) --
(59/25, 8/13) --
(71/30, 7/13) --
(178/75, 8/13) --
(119/50, 7/13) --
(179/75, 8/13) --
(359/150, 9/13) --
(12/5, 8/13) --
(361/150, 9/13) --
(181/75, 8/13) --
(121/50, 9/13) --
(182/75, 10/13) --
(73/30, 9/13) --
(61/25, 8/13) --
(367/150, 7/13) --
(184/75, 8/13) --
(123/50, 7/13) --
(37/15, 6/13) --
(371/150, 5/13) --
(62/25, 4/13) --
(373/150, 5/13) --
(187/75, 6/13) --
(5/2, 7/13) --
(188/75, 8/13) --
(377/150, 7/13) --
(63/25, 8/13) --
(379/150, 9/13) --
(38/15, 10/13) --
(127/50, 9/13) --
(191/75, 10/13) --
(383/150, 9/13) --
(64/25, 8/13) --
(77/30, 7/13) --
(193/75, 6/13) --
(129/50, 5/13) --
(194/75, 6/13) --
(389/150, 7/13) --
(13/5, 6/13) --
(391/150, 5/13) --
(196/75, 6/13) --
(131/50, 7/13) --
(197/75, 6/13) --
(79/30, 5/13) --
(66/25, 4/13) --
(397/150, 5/13) --
(199/75, 4/13) --
(133/50, 3/13) --
(8/3, 4/13) --
(401/150, 5/13) --
(67/25, 6/13) --
(403/150, 7/13) --
(202/75, 6/13) --
(27/10, 5/13) --
(203/75, 6/13) --
(407/150, 7/13) --
(68/25, 6/13) --
(409/150, 5/13) --
(41/15, 4/13) --
(137/50, 5/13) --
(206/75, 4/13) --
(413/150, 5/13) --
(69/25, 6/13) --
(83/30, 7/13) --
(208/75, 8/13) --
(139/50, 9/13) --
(209/75, 8/13) --
(419/150, 7/13) --
(14/5, 8/13) --
(421/150, 7/13) --
(211/75, 6/13) --
(141/50, 7/13) --
(212/75, 8/13) --
(17/6, 7/13) --
(71/25, 6/13) --
(427/150, 5/13) --
(214/75, 6/13) --
(143/50, 7/13) --
(43/15, 6/13) --
(431/150, 7/13) --
(72/25, 6/13) --
(433/150, 7/13) --
(217/75, 6/13) --
(29/10, 5/13) --
(218/75, 4/13) --
(437/150, 5/13) --
(73/25, 4/13) --
(439/150, 5/13) --
(44/15, 6/13) --
(147/50, 7/13) --
(221/75, 8/13) --
(443/150, 7/13) --
(74/25, 6/13) --
(89/30, 5/13) --
(223/75, 4/13) --
(149/50, 5/13) --
(224/75, 6/13) --
(449/150, 7/13) --
(3, 8/13) --
(451/150, 9/13) --
(226/75, 10/13) --
(151/50, 9/13) --
(227/75, 8/13) --
(91/30, 9/13) --
(76/25, 10/13) --
(457/150, 9/13) --
(229/75, 8/13) --
(153/50, 9/13) --
(46/15, 10/13) --
(461/150, 9/13) --
(77/25, 10/13) --
(463/150, 9/13) --
(232/75, 10/13) --
(31/10, 11/13) --
(233/75, 10/13) --
(467/150, 9/13) --
(78/25, 10/13) --
(469/150, 11/13) --
(47/15, 12/13) --
(157/50, 11/13) --
(236/75, 10/13) --
(473/150, 11/13) --
(79/25, 12/13) --
(19/6, 1) --
(238/75, 14/13) --
(159/50, 1) --
(239/75, 14/13) --
(479/150, 1) --
(16/5, 12/13) --
(481/150, 11/13) --
(241/75, 12/13) --
(161/50, 11/13) --
(242/75, 12/13) --
(97/30, 1) --
(81/25, 12/13) --
(487/150, 1) --
(244/75, 14/13) --
(163/50, 15/13) --
(49/15, 16/13) --
(491/150, 17/13) --
(82/25, 18/13) --
(493/150, 17/13) --
(247/75, 18/13) --
(33/10, 19/13) --
(248/75, 20/13) --
(497/150, 21/13) --
(83/25, 20/13) --
(499/150, 19/13) --
(10/3, 18/13) --
(167/50, 17/13) --
(251/75, 16/13) --
(503/150, 17/13) --
(84/25, 16/13) --
(101/30, 15/13) --
(253/75, 14/13) --
(169/50, 1) --
(254/75, 12/13) --
(509/150, 11/13) --
(17/5, 12/13) --
(511/150, 1) --
(256/75, 12/13) --
(171/50, 11/13) --
(257/75, 10/13) --
(103/30, 11/13) --
(86/25, 12/13) --
(517/150, 11/13) --
(259/75, 10/13) --
(173/50, 9/13) --
(52/15, 8/13) --
(521/150, 9/13) --
(87/25, 10/13) --
(523/150, 11/13) --
(262/75, 10/13) --
(7/2, 9/13) --
(263/75, 8/13) --
(527/150, 7/13) --
(88/25, 8/13) --
(529/150, 7/13) --
(53/15, 6/13) --
(177/50, 7/13) --
(266/75, 8/13) --
(533/150, 9/13) --
(89/25, 8/13) --
(107/30, 7/13) --
(268/75, 8/13) --
(179/50, 7/13) --
(269/75, 8/13) --
(539/150, 9/13) --
(18/5, 10/13) --
(541/150, 11/13) --
(271/75, 10/13) --
(181/50, 11/13) --
(272/75, 12/13) --
(109/30, 1) --
(91/25, 14/13) --
(547/150, 1) --
(274/75, 12/13) --
(183/50, 1) --
(11/3, 12/13) --
(551/150, 1) --
(92/25, 14/13) --
(553/150, 15/13) --
(277/75, 16/13) --
(37/10, 15/13) --
(278/75, 16/13) --
(557/150, 17/13) --
(93/25, 18/13) --
(559/150, 17/13) --
(56/15, 16/13) --
(187/50, 17/13) --
(281/75, 18/13) --
(563/150, 19/13) --
(94/25, 18/13) --
(113/30, 17/13) --
(283/75, 18/13) --
(189/50, 19/13) --
(284/75, 20/13) --
(569/150, 19/13) --
(19/5, 20/13) --
(571/150, 21/13) --
(286/75, 22/13) --
(191/50, 23/13) --
(287/75, 24/13) --
(23/6, 23/13) --
(96/25, 22/13) --
(577/150, 23/13) --
(289/75, 22/13) --
(193/50, 21/13) --
(58/15, 22/13) --
(581/150, 23/13) --
(97/25, 24/13) --
(583/150, 23/13) --
(292/75, 24/13) --
(39/10, 23/13) --
(293/75, 24/13) --
(587/150, 25/13) --
(98/25, 24/13) --
(589/150, 25/13) --
(59/15, 24/13) --
(197/50, 23/13) --
(296/75, 24/13) --
(593/150, 23/13) --
(99/25, 22/13) --
(119/30, 21/13) --
(298/75, 22/13) --
(199/50, 23/13) --
(299/75, 24/13) --
(599/150, 23/13) --
(4, 24/13) --
(601/150, 25/13) --
(301/75, 2) --
(201/50, 25/13) --
(302/75, 2) --
(121/30, 25/13) --
(101/25, 24/13) --
(607/150, 25/13) --
(304/75, 24/13) --
(203/50, 23/13) --
(61/15, 24/13) --
(611/150, 25/13) --
(102/25, 24/13) --
(613/150, 25/13) --
(307/75, 2) --
(41/10, 27/13) --
(308/75, 28/13) --
(617/150, 27/13) --
(103/25, 2) --
(619/150, 27/13) --
(62/15, 2) --
(207/50, 25/13) --
(311/75, 24/13) --
(623/150, 23/13) --
(104/25, 24/13) --
(25/6, 23/13) --
(313/75, 22/13) --
(209/50, 21/13) --
(314/75, 20/13) --
(629/150, 19/13) --
(21/5, 18/13) --
(631/150, 19/13) --
(316/75, 18/13) --
(211/50, 19/13) --
(317/75, 20/13) --
(127/30, 21/13) --
(106/25, 20/13) --
(637/150, 21/13) --
(319/75, 20/13) --
(213/50, 19/13) --
(64/15, 18/13) --
(641/150, 17/13) --
(107/25, 16/13) --
(643/150, 15/13) --
(322/75, 16/13) --
(43/10, 15/13) --
(323/75, 16/13) --
(647/150, 17/13) --
(108/25, 16/13) --
(649/150, 17/13) --
(13/3, 18/13) --
(217/50, 17/13) --
(326/75, 16/13) --
(653/150, 15/13) --
(109/25, 16/13) --
(131/30, 17/13) --
(328/75, 18/13) --
(219/50, 19/13) --
(329/75, 20/13) --
(659/150, 19/13) --
(22/5, 18/13) --
(661/150, 19/13) --
(331/75, 18/13) --
(221/50, 17/13) --
(332/75, 16/13) --
(133/30, 17/13) --
(111/25, 18/13) --
(667/150, 19/13) --
(334/75, 18/13) --
(223/50, 19/13) --
(67/15, 20/13) --
(671/150, 21/13) --
(112/25, 20/13) --
(673/150, 19/13) --
(337/75, 20/13) --
(9/2, 21/13) --
(338/75, 22/13) --
(677/150, 21/13) --
(113/25, 20/13) --
(679/150, 19/13) --
(68/15, 18/13) --
(227/50, 19/13) --
(341/75, 20/13) --
(683/150, 19/13) --
(114/25, 20/13) --
(137/30, 21/13) --
(343/75, 20/13) --
(229/50, 19/13) --
(344/75, 20/13) --
(689/150, 21/13) --
(23/5, 22/13) --
(691/150, 21/13) --
(346/75, 22/13) --
(231/50, 23/13) --
(347/75, 24/13) --
(139/30, 23/13) --
(116/25, 24/13) --
(697/150, 23/13) --
(349/75, 22/13) --
(233/50, 21/13) --
(14/3, 22/13) --
(701/150, 21/13) --
(117/25, 22/13) --
(703/150, 21/13) --
(352/75, 22/13) --
(47/10, 21/13) --
(353/75, 20/13) --
(707/150, 19/13) --
(118/25, 18/13) --
(709/150, 17/13) --
(71/15, 16/13) --
(237/50, 15/13) --
(356/75, 16/13) --
(713/150, 17/13) --
(119/25, 18/13) --
(143/30, 17/13) --
(358/75, 18/13) --
(239/50, 19/13) --
(359/75, 20/13) --
(719/150, 19/13) --
(24/5, 18/13) --
(721/150, 17/13) --
(361/75, 18/13) --
(241/50, 19/13) --
(362/75, 20/13) --
(29/6, 19/13) --
(121/25, 18/13) --
(727/150, 19/13) --
(364/75, 20/13) --
(243/50, 19/13) --
(73/15, 20/13) --
(731/150, 19/13) --
(122/25, 18/13) --
(733/150, 17/13) --
(367/75, 18/13) --
(49/10, 19/13) --
(368/75, 18/13) --
(737/150, 19/13) --
(123/25, 20/13) --
(739/150, 19/13) --
(74/15, 20/13) --
(247/50, 19/13) --
(371/75, 18/13) --
(743/150, 17/13) --
(124/25, 18/13) --
(149/30, 19/13) --
(373/75, 20/13) --
(249/50, 21/13) --
(374/75, 22/13) --
(749/150, 21/13) --
(5, 20/13) --
(751/150, 21/13) --
(376/75, 20/13) --
(251/50, 21/13) --
(377/75, 22/13) --
(151/30, 21/13) --
(126/25, 20/13) --
(757/150, 21/13) --
(379/75, 22/13) --
(253/50, 21/13) --
(76/15, 22/13) --
(761/150, 21/13) --
(127/25, 20/13) --
(763/150, 21/13) --
(382/75, 20/13) --
(51/10, 21/13) --
(383/75, 22/13) --
(767/150, 21/13) --
(128/25, 22/13) --
(769/150, 23/13) --
(77/15, 24/13) --
(257/50, 23/13) --
(386/75, 22/13) --
(773/150, 23/13) --
(129/25, 22/13) --
(31/6, 23/13) --
(388/75, 22/13) --
(259/50, 23/13) --
(389/75, 24/13) --
(779/150, 25/13) --
(26/5, 2) --
(781/150, 27/13) --
(391/75, 2) --
(261/50, 25/13) --
(392/75, 24/13) --
(157/30, 25/13) --
(131/25, 2) --
(787/150, 27/13) --
(394/75, 28/13) --
(263/50, 29/13) --
(79/15, 30/13) --
(791/150, 29/13) --
(132/25, 30/13) --
(793/150, 29/13) --
(397/75, 28/13) --
(53/10, 27/13) --
(398/75, 28/13) --
(797/150, 27/13) --
(133/25, 28/13) --
(799/150, 27/13) --
(16/3, 2) --
(267/50, 27/13) --
(401/75, 2) --
(803/150, 25/13) --
(134/25, 24/13) --
(161/30, 23/13) --
(403/75, 22/13) --
(269/50, 23/13) --
(404/75, 24/13) --
(809/150, 23/13) --
(27/5, 22/13) --
(811/150, 23/13) --
(406/75, 22/13) --
(271/50, 21/13) --
(407/75, 22/13) --
(163/30, 21/13) --
(136/25, 22/13) --
(817/150, 21/13) --
(409/75, 20/13) --
(273/50, 21/13) --
(82/15, 20/13) --
(821/150, 21/13) --
(137/25, 20/13) --
(823/150, 21/13) --
(412/75, 20/13) --
(11/2, 19/13) --
(413/75, 18/13) --
(827/150, 17/13) --
(138/25, 16/13) --
(829/150, 15/13) --
(83/15, 16/13) --
(277/50, 15/13) --
(416/75, 14/13) --
(833/150, 1) --
(139/25, 14/13) --
(167/30, 15/13) --
(418/75, 14/13) --
(279/50, 15/13) --
(419/75, 16/13) --
(839/150, 15/13) --
(28/5, 16/13) --
(841/150, 15/13) --
(421/75, 16/13) --
(281/50, 17/13) --
(422/75, 18/13) --
(169/30, 19/13) --
(141/25, 18/13) --
(847/150, 19/13) --
(424/75, 18/13) --
(283/50, 17/13) --
(17/3, 16/13) --
(851/150, 17/13) --
(142/25, 18/13) --
(853/150, 19/13) --
(427/75, 20/13) --
(57/10, 19/13) --
(428/75, 18/13) --
(857/150, 17/13) --
(143/25, 18/13) --
(859/150, 17/13) --
(86/15, 16/13) --
(287/50, 17/13) --
(431/75, 16/13) --
(863/150, 15/13) --
(144/25, 16/13) --
(173/30, 15/13) --
(433/75, 16/13) --
(289/50, 17/13) --
(434/75, 16/13) --
(869/150, 17/13) --
(29/5, 16/13) --
(871/150, 15/13) --
(436/75, 14/13) --
(291/50, 15/13) --
(437/75, 16/13) --
(35/6, 17/13) --
(146/25, 18/13) --
(877/150, 17/13) --
(439/75, 16/13) --
(293/50, 15/13) --
(88/15, 14/13) --
(881/150, 15/13) --
(147/25, 16/13) --
(883/150, 15/13) --
(442/75, 14/13) --
(59/10, 1) --
(443/75, 12/13) --
(887/150, 11/13) --
(148/25, 10/13) --
(889/150, 11/13) --
(89/15, 10/13) --
(297/50, 11/13) --
(446/75, 12/13) --
(893/150, 11/13) --
(149/25, 10/13) --
(179/30, 9/13) --
(448/75, 10/13) --
(299/50, 11/13) --
(449/75, 10/13) --
(899/150, 11/13) --
(6, 12/13) --
(901/150, 11/13) --
(451/75, 10/13) --
(301/50, 9/13) --
(452/75, 10/13) --
(181/30, 11/13) --
(151/25, 10/13) --
(907/150, 9/13) --
(454/75, 10/13) --
(303/50, 11/13) --
(91/15, 10/13) --
(911/150, 9/13) --
(152/25, 8/13) --
(913/150, 9/13) --
(457/75, 10/13) --
(61/10, 9/13) --
(458/75, 8/13) --
(917/150, 9/13) --
(153/25, 10/13) --
(919/150, 9/13) --
(92/15, 8/13) --
(307/50, 9/13) --
(461/75, 8/13) --
(923/150, 9/13) --
(154/25, 8/13) --
(37/6, 9/13) --
(463/75, 8/13) --
(309/50, 9/13) --
(464/75, 10/13) --
(929/150, 9/13) --
(31/5, 10/13) --
(931/150, 9/13) --
(466/75, 8/13) --
(311/50, 7/13) --
(467/75, 8/13) --
(187/30, 7/13) --
(156/25, 8/13) --
(937/150, 9/13) --
(469/75, 8/13) --
(313/50, 9/13) --
(94/15, 10/13) --
(941/150, 11/13) --
(157/25, 12/13) --
(943/150, 11/13) --
(472/75, 12/13) --
(63/10, 11/13) --
(473/75, 10/13) --
(947/150, 11/13) --
(158/25, 12/13) --
(949/150, 11/13) --
(19/3, 12/13) --
(317/50, 11/13) --
(476/75, 12/13) --
(953/150, 11/13) --
(159/25, 12/13) --
(191/30, 11/13) --
(478/75, 10/13) --
(319/50, 9/13) --
(479/75, 10/13) --
(959/150, 9/13) --
(32/5, 10/13) --
(961/150, 9/13) --
(481/75, 10/13) --
(321/50, 9/13) --
(482/75, 8/13) --
(193/30, 7/13) --
(161/25, 6/13) --
(967/150, 7/13) --
(484/75, 6/13) --
(323/50, 5/13) --
(97/15, 4/13) --
(971/150, 5/13) --
(162/25, 4/13) --
(973/150, 3/13) --
(487/75, 4/13) --
(13/2, 3/13) --
(488/75, 2/13) --
(977/150, 1/13) --
(163/25, 2/13) --
(979/150, 3/13) --
(98/15, 4/13) --
(327/50, 3/13) --
(491/75, 2/13) --
(983/150, 3/13) --
(164/25, 2/13) --
(197/30, 1/13) --
(493/75, 2/13) --
(329/50, 3/13) --
(494/75, 2/13) --
(989/150, 3/13) --
(33/5, 2/13) --
(991/150, 1/13) --
(496/75, 2/13) --
(331/50, 3/13) --
(497/75, 2/13) --
(199/30, 1/13) --
(166/25, 2/13) --
(997/150, 1/13) --
(499/75, 2/13) --
(333/50, 1/13) --
(20/3, 0); 

\draw[ultra thin](526/75, 0) --
(527/75, 0) --
(176/25, 0) --
(177/25, 2/13) --
(532/75, 2/13) --
(107/15, 3/13) --
(536/75, 1/13) --
(179/25, 4/13) --
(538/75, 4/13) --
(36/5, 4/13) --
(181/25, 3/13) --
(544/75, 6/13) --
(109/15, 8/13) --
(182/25, 7/13) --
(547/75, 6/13) --
(22/3, 6/13) --
(551/75, 9/13) --
(184/25, 9/13) --
(553/75, 5/13) --
(554/75, 9/13) --
(37/5, 8/13) --
(556/75, 10/13) --
(186/25, 3/13) --
(559/75, 7/13) --
(112/15, 7/13) --
(187/25, 7/13) --
(563/75, 7/13) --
(113/15, 8/13) --
(566/75, 8/13) --
(189/25, 8/13) --
(568/75, 7/13) --
(569/75, 9/13) --
(571/75, 6/13) --
(572/75, 10/13) --
(191/25, 9/13) --
(574/75, 11/13) --
(23/3, 9/13) --
(192/25, 9/13) --
(577/75, 8/13) --
(578/75, 10/13) --
(193/25, 7/13) --
(116/15, 9/13) --
(581/75, 4/13) --
(194/25, 7/13) --
(583/75, 9/13) --
(584/75, 11/13) --
(39/5, 9/13) --
(586/75, 8/13) --
(587/75, 5/13) --
(196/25, 4/13) --
(589/75, 4/13) --
(596/75, 4/13) --
(599/75, 3/13) --
(8, 5/13) --
(601/75, 4/13) --
(201/25, 4/13) --
(604/75, 6/13) --
(121/15, 4/13) --
(202/25, 4/13) --
(41/5, 9/13) --
(206/25, 11/13) --
(619/75, 10/13) --
(124/15, 12/13) --
(207/25, 11/13) --
(622/75, 11/13) --
(623/75, 11/13) --
(208/25, 10/13) --
(25/3, 8/13) --
(626/75, 6/13) --
(628/75, 7/13) --
(42/5, 3/13) --
(631/75, 9/13) --
(632/75, 8/13) --
(211/25, 10/13) --
(634/75, 9/13) --
(127/15, 9/13) --
(212/25, 4/13) --
(637/75, 8/13) --
(638/75, 10/13) --
(213/25, 9/13) --
(128/15, 9/13) --
(641/75, 8/13) --
(643/75, 9/13) --
(644/75, 9/13) --
(43/5, 7/13) --
(646/75, 10/13) --
(647/75, 10/13) --
(649/75, 8/13) --
(26/3, 10/13) --
(217/25, 12/13) --
(653/75, 11/13) --
(219/25, 14/13) --
(659/75, 14/13) --
(44/5, 16/13) --
(661/75, 15/13) --
(662/75, 1) --
(221/25, 16/13) --
(664/75, 16/13) --
(133/15, 14/13) --
(222/25, 14/13) --
(667/75, 12/13) --
(668/75, 14/13) --
(223/25, 1) --
(671/75, 1) --
(224/25, 15/13) --
(673/75, 14/13) --
(9, 15/13) --
(676/75, 15/13) --
(679/75, 17/13) --
(136/15, 1) --
(227/25, 17/13) --
(682/75, 19/13) --
(683/75, 18/13) --
(228/25, 18/13) --
(137/15, 14/13) --
(686/75, 16/13) --
(229/25, 15/13) --
(689/75, 16/13) --
(46/5, 11/13) --
(691/75, 1) --
(692/75, 18/13) --
(231/25, 16/13) --
(694/75, 18/13) --
(139/15, 20/13) --
(232/25, 12/13) --
(697/75, 8/13) --
(28/3, 10/13) --
(701/75, 7/13) --
(234/25, 11/13) --
(703/75, 10/13) --
(704/75, 12/13) --
(47/5, 10/13) --
(706/75, 10/13) --
(707/75, 7/13) --
(236/25, 3/13) --
(709/75, 8/13) --
(142/15, 8/13) --
(237/25, 7/13) --
(712/75, 9/13) --
(713/75, 7/13) --
(239/25, 7/13) --
(718/75, 4/13) --
(719/75, 6/13) --
(48/5, 9/13) --
(721/75, 9/13) --
(722/75, 4/13) --
(241/25, 6/13) --
(724/75, 4/13) --
(29/3, 6/13) --
(242/25, 4/13) --
(243/25, 4/13) --
(146/15, 6/13) --
(731/75, 4/13) --
(244/25, 6/13) --
(733/75, 4/13) --
(736/75, 3/13) --
(737/75, 6/13) --
(246/25, 8/13) --
(739/75, 7/13) --
(148/15, 5/13) --
(247/25, 7/13) --
(743/75, 6/13) --
(248/25, 6/13) --
(149/15, 6/13) --
(746/75, 4/13) --
(748/75, 3/13) --
(749/75, 5/13) --
(10, 7/13) --
(151/15, 7/13) --
(252/25, 9/13) --
(757/75, 7/13) --
(758/75, 9/13) --
(152/15, 9/13) --
(761/75, 9/13) --
(254/25, 8/13) --
(763/75, 10/13) --
(51/5, 9/13) --
(766/75, 11/13) --
(256/25, 8/13) --
(769/75, 10/13) --
(154/15, 1) --
(257/25, 1) --
(772/75, 11/13) --
(258/25, 12/13) --
(779/75, 11/13) --
(52/5, 17/13) --
(781/75, 14/13) --
(782/75, 16/13) --
(261/25, 18/13) --
(784/75, 20/13) --
(157/15, 16/13) --
(262/25, 10/13) --
(787/75, 12/13) --
(788/75, 9/13) --
(263/25, 11/13) --
(158/15, 6/13) --
(791/75, 8/13) --
(264/25, 10/13) --
(793/75, 7/13) --
(53/5, 6/13) --
(796/75, 8/13) --
(797/75, 7/13) --
(267/25, 10/13) --
(268/25, 11/13) --
(161/15, 1) --
(806/75, 12/13) --
(54/5, 15/13) --
(812/75, 15/13) --
(271/25, 17/13) --
(163/15, 16/13) --
(272/25, 18/13) --
(273/25, 19/13) --
(274/25, 18/13) --
(823/75, 21/13) --
(824/75, 23/13) --
(11, 22/13) --
(276/25, 23/13) --
(829/75, 23/13) --
(166/15, 21/13) --
(277/25, 24/13) --
(832/75, 24/13) --
(833/75, 23/13) --
(836/75, 23/13) --
(279/25, 19/13) --
(838/75, 22/13) --
(839/75, 25/13) --
(56/5, 25/13) --
(841/75, 24/13) --
(281/25, 24/13) --
(844/75, 20/13) --
(169/15, 22/13) --
(282/25, 25/13) --
(847/75, 27/13) --
(848/75, 2) --
(283/25, 23/13) --
(34/3, 18/13) --
(851/75, 15/13) --
(284/25, 17/13) --
(853/75, 20/13) --
(854/75, 20/13) --
(57/5, 15/13) --
(856/75, 12/13) --
(857/75, 16/13) --
(286/25, 15/13) --
(859/75, 17/13) --
(287/25, 12/13) --
(862/75, 14/13) --
(863/75, 17/13) --
(288/25, 19/13) --
(173/15, 18/13) --
(868/75, 18/13) --
(869/75, 1) --
(58/5, 18/13) --
(871/75, 20/13) --
(872/75, 17/13) --
(291/25, 19/13) --
(874/75, 21/13) --
(292/25, 19/13) --
(877/75, 18/13) --
(878/75, 20/13) --
(881/75, 21/13) --
(294/25, 18/13) --
(883/75, 20/13) --
(884/75, 23/13) --
(59/5, 23/13) --
(886/75, 21/13) --
(887/75, 21/13) --
(296/25, 21/13) --
(297/25, 17/13) --
(892/75, 15/13) --
(893/75, 17/13) --
(298/25, 19/13) --
(179/15, 12/13) --
(896/75, 17/13) --
(299/25, 19/13) --
(898/75, 16/13) --
(899/75, 19/13) --
(12, 19/13) --
(901/75, 1) --
(902/75, 18/13) --
(301/25, 16/13) --
(904/75, 19/13) --
(181/15, 19/13) --
(303/25, 19/13) --
(182/15, 21/13) --
(911/75, 20/13) --
(304/25, 19/13) --
(913/75, 21/13) --
(914/75, 18/13) --
(61/5, 21/13) --
(916/75, 21/13) --
(917/75, 20/13) --
(919/75, 21/13) --
(307/25, 21/13) --
(922/75, 23/13) --
(923/75, 22/13) --
(308/25, 22/13) --
(309/25, 22/13) --
(928/75, 24/13) --
(929/75, 2) --
(311/25, 24/13) --
(934/75, 2) --
(187/15, 29/13) --
(312/25, 29/13) --
(937/75, 27/13) --
(938/75, 27/13) --
(313/25, 2) --
(188/15, 21/13) --
(941/75, 23/13) --
(314/25, 22/13) --
(943/75, 21/13) --
(944/75, 21/13) --
(63/5, 20/13) --
(946/75, 20/13) --
(947/75, 20/13) --
(316/25, 15/13) --
(38/3, 14/13) --
(952/75, 15/13) --
(953/75, 15/13) --
(191/15, 1) --
(956/75, 15/13) --
(319/25, 18/13) --
(958/75, 18/13) --
(959/75, 12/13) --
(64/5, 15/13) --
(961/75, 17/13) --
(962/75, 19/13) --
(321/25, 17/13) --
(964/75, 16/13) --
(193/15, 15/13) --
(322/25, 1) --
(967/75, 16/13) --
(968/75, 16/13) --
(323/25, 11/13) --
(194/15, 1) --
(971/75, 15/13) --
(324/25, 17/13) --
(973/75, 1) --
(974/75, 15/13) --
(13, 10/13) --
(976/75, 9/13) --
(977/75, 11/13) --
(326/25, 6/13) --
(979/75, 10/13) --
(196/15, 9/13) --
(327/25, 11/13) --
(982/75, 8/13) --
(983/75, 10/13) --
(328/25, 8/13) --
(197/15, 10/13) --
(986/75, 7/13) --
(329/25, 9/13) --
(988/75, 7/13) --
(989/75, 9/13) --
(66/5, 8/13) --
(991/75, 8/13) --
(992/75, 8/13) --
(331/25, 6/13) --
(994/75, 9/13) --
(199/15, 9/13) --
(332/25, 7/13) --
(998/75, 8/13) --
(1001/75, 8/13) --
(334/25, 11/13) --
(1003/75, 11/13) --
(1004/75, 6/13) --
(67/5, 11/13) --
(1006/75, 11/13) --
(1007/75, 11/13) --
(336/25, 11/13) --
(1009/75, 9/13) --
(202/15, 9/13) --
(337/25, 9/13) --
(1012/75, 6/13) --
(1013/75, 4/13) --
(338/25, 3/13) --
(339/25, 3/13) --
(1018/75, 2/13) --
(68/5, 2/13) --
(1021/75, 2/13) --
(1022/75, 0) --
(341/25, 2/13) --
(1024/75, 1/13) --
(41/3, 1/13) --
(41/3, 0)
;

\draw (0,3) -> (0,0) -> (6.75,0 );
\draw (7,3) -> (7,0) -> (13.75,0 );
\end{tikzpicture}
\caption{ $\Gamma^n(2nt)/\sqrt{2n}$ along with $-F_n$}

\label{fig 231}
\end{figure}
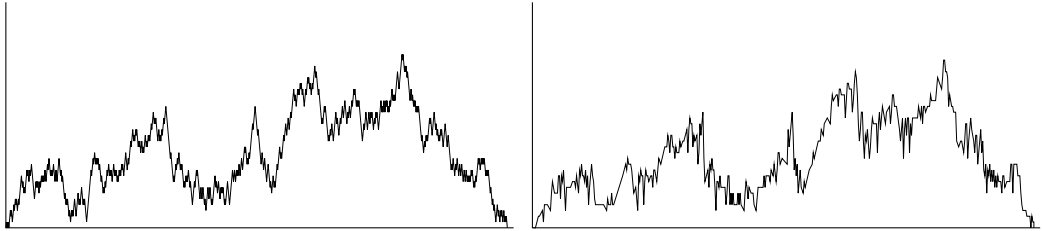

\subsection*{Fixed points of \textbf{231}-avoiding permutations}
The situation for $\mathbf{231}$-avoiding permutations is quite different from the case of $\mathbf{123}$-avoiding permutations.  In order to explain the connection with the invariance principle from \cite{part1}, we make use of the bijection we define below in \eqref{231 bijection}.  The details of this bijection are not needed for the present discussion, but will be used later in the paper.  The following result is the invariance principle we obtain from \cite{part1}, see also Figure \ref{fig 231}.

\begin{theorem}{\cite[Theorem 1.3]{part1}} \label{london}
Let $\Gamma^{n}$ be a uniformly chosen Dyck path of length $2n$ and let $\sigma_{\Gamma^n}$ be the image of $\Gamma^n$ under the bijection \eqref{231 bijection}, so that $\sigma_{\Gamma^n}$  is a uniformly random $\textbf{231}$-avoiding permutation.
For any $\epsilon>0$ there exists a sequence of sets $\shortexc_{\Gamma^n}$ such that
$$\P\left(|\shortexc_{\Gamma^n}|>n-n^{.75+\epsilon}\right) \to 1$$ and 
$$ \left(\frac{\Gamma^n(2nt)}{\sqrt{2n}}, \frac{F_{n}(nt) - nt}{\sqrt{2n}}\right)_{t\in [0,1]} \Rightarrow
\left( \mathbbm{e}_t,-\mathbbm{e}_t\right)_{t\in [0,1]},$$
where $F_n$ is the linear interpolation of the points 
\[  \left\{ \left(i,\sigma_{\Gamma^n}(i)\right) : i\in \shortexc_{\Gamma^n} \right\} \cup \{(0,0)\} \cup \{(n+1,n+1)\}\]
\end{theorem}

This theorem shows that the bulk of the points in a uniformly random $\textbf{231}$-avoiding permutation closely follow a Brownian excursion.  However, most of the fixed points of $\sigma_{\Gamma^n}$ are in the set ($[n]\setminus \shortexc_{\Gamma^n}$) of exceptional points that Theorem \ref{london} does not provide much information about.
Nonetheless, we can still describe the asymptotic distribution of fixed point in terms of the limiting excursion of the Dyck path $\Gamma^n$.  Our description of these points will  allow us to greatly generalize the results in \cite{mp} about distribution of a random $\textbf{231}$-avoiding permutation close to the diagonal. 
We count the number of fixed points in an interval by
$$\fp_{[an,bn]}(\sigma)=|\{i \in[an,bn]:\ \sigma(i)=i\}|$$ where $0\leq a <b\leq1.$
Miner and Pak proved that the expected number of fixed points in an interval of the form $[1,n]$ is of order $n^{1/4}$ \cite{mp}. This next theorem shows that a typical \textbf{231}-avoiding permutation has on the order of $n^{1/4}$ fixed points. Moreover it allows us to calculate the distribution of
$$\frac{1}{n^{1/4}}\fp_{[an,bn]}(\sigma_{\Gamma^n}). $$

We now state of version of Theorem \ref{intro main} Part (a) that gives joint convergence of the Dyck path and the fixed points of the associated $\mathbf{231}$-avoiding permutation.

\begin{theorem}\label{geddewatanabe}
Let $\Gamma^n$ be a uniformly random Dyck path of length $2n$ and let $(\Gamma^n(t),0\leq t\leq 2n)$ be its linear interpolation.  We then have the joint convergence
\[\left(\frac{\Gamma^n(2ns)}{\sqrt{2n}}, \frac{\fp_{[0,tn]}(\sigma_{\Gamma^n})}{n^{1/4}}\right)_{(s,t)\in [0,1]^2} \overset{dist}{\longrightarrow} \left( \mathbbm{e}_s , \int_0^t \frac{1}{2^{7/4}\pi^{1/2}\mathbbm{e}_u^{3/2}} du\right)_{(s,t)\in [0,1]^2}\]
in distribution on $D([0,1],\R)\times D([0,1],\R)$, where $(\mathbbm{e}_t,0\leq t\leq 1)$ is Brownian excursion and $D([0,1],\R)$ is the space of right continuous functions with left limits equipped with the Skorokhod topology.
\end{theorem} 

\begin{remark}
We may replace $D([0,1],\R)$ with $C([0,1],\R)$ in the preceding theorem if we replace $\fp_{[0,sn]}(\sigma_{\Gamma^n})$ with the linear interpolation of the points $(\fp_{[0,i]}(\sigma_{\Gamma^n}), i=0,\dots,n)$.
\end{remark}

\begin{remark}
Theorem \ref{geddewatanabe} does more than just tell us the distribution of fixed points on an interval. It relates (with high probability) the number of fixed points of $\sigma_{\Gamma^n}$ on an interval $[an,bn]$ with the shape of $\Gamma^n$  on the same interval.
Roughly speaking if we know a Brownian excursion that approximates a scaled Dyck path then with high probability we can quite closely determine the number of fixed points for the corresponding $\textbf{231}$-avoiding permutation. Also if we know the location of fixed points for a $\textbf{231}$-avoiding permutation then we can use that to find a Brownian excursion that does a good job approximating the corresponding scaled Dyck path. 
\end{remark}

\subsection*{``Almost fixed points"}
Perhaps the most interesting result in \cite{mp} is a phase transition it shows in 
$$\P\left(\sigma(i)=\left \lfloor i-\left(\frac{(i(n-i)}{n}\right)^{\alpha}\right \rfloor\right)$$
that occurs at $\alpha=3/8$. 
In particular they show that
\begin{equation} \label{paris}
\P\left(\sigma(i)=\left \lfloor i-\left(\frac{(i(n-i)}{n}\right)^{\alpha}\right \rfloor \right) \sim
\begin{cases} Cn^{-3/4} & \text{if $\alpha \in (0,3/8)$} \\
Cn^{-(3/2-2\alpha)}& \text{if $\alpha \in [3/8,.5)$} 
\end{cases} 
\end{equation}

This result is particularly intriguing because it is not clear what is driving the phase transition. Miner and Pak say that their results on \textbf{231}-avoiding permutations ``are extremely unusual, and have yet to be explained even on a qualitative level" \cite{mp}. In this paper we use a generalization of Theorem \ref{geddewatanabe} to give an explanation of these results. 

 First we show that the difference is not due to the number of ``almost" fixed points on a typical path. To make this precise we define
$$\fp^{K,\alpha}_{[an,bn]}(\sigma)= \left|\left\{i:\ \sigma(i)=i-\left \lfloor K\left( \frac{i(n-i)}{n}\right)^\alpha\right \rfloor\right\} \cap[an,bn]\right|.$$
Then we follow the proof of Theorem \ref{geddewatanabe} very closely to show

\begin{corollary}\label{sixteencandles}
Let $\Gamma^n$ be a uniformly random Dyck path of length $2n$ and let $(\Gamma^n(t),0\leq t\leq 2n)$ be its linear interpolation.  Fix any $0<a<b<1$. We then have the joint convergence
\[\left(\frac{\Gamma^n(2ns)}{\sqrt{2n}}, \frac{\fp^{K,\alpha}_{[an,tn]}(\sigma_{\Gamma^n})}{n^{1/4}}\right)_{(s,t)\in [a,b]^2} \overset{dist}{\longrightarrow} \left( \mathbbm{e}_s , \int_a^t \frac{1}{2^{7/4}\pi^{1/2}\mathbbm{e}_u^{3/2}} du\right)_{(s,t)\in [a,b]^2}\]
in distribution on $D([a,b],\R)\times D([a,b],\R)$, where $(\mathbbm{e}_t,0\leq t\leq 1)$ is Brownian excursion and $D([a,b],\R)$ is the space of right continuous functions with left limits equipped with the Skorokhod topology.
\end{corollary}

Thus for all $K$ and $\alpha$ the distribution of the number of ``almost" fixed points is asymptotically the same as the distribution of the number of fixed points and we do not see the same phase transition that Miner and Pak observed. 

But there is no inconsistency between our results and \cite{mp} in the regime $K>0$ and $\alpha \in [3/8,.5)$. This is because a small number of permutations drive the probability that Miner and Pak calculate in (\ref{paris}). This is missed by our convergence in distribution.
This small number of permutations are the ones $\sigma_{\gamma}$ whose corresponding Dyck paths $\gamma$ have height $\gamma(i)=\lfloor K(i(n-i)/n)^\alpha\rfloor$ for some $i \in [2an,2bn]$. 
As the density of these permutations becomes vanishingly small as $n \to \infty$, these permutations do not affect the limiting distribution of $n^{-1/4}\fp^{K,\alpha}_{[an,bn]}(\sigma)$ that we calculate. But these are the permutations that dominate the probabilities that Miner and Pak calculate.

\section{321-avoiding permutations}
\label{offer}

We now describe a bijection (which is often known as the Billey-Jockusch-Stanley or BJS bijection) from Dyck paths of length $2n$ to \textbf{321}-avoiding permutations of length $n$ \cite{callan}.
Fix a Dyck path $\gamma:\{0,1,\dots, 2n\} \to\N$ of length $2n$. Given $\gamma$ define the following.  Let $m$ be the number of runs of increases (or decreases) in $\gamma$. Let $a_i$ be the number of increases in the $i$th run of increases in $\gamma$. Let $A_i=\sum_{j=1}^ia_j$ and
let $\A = \cup_{i=1}^{m-1}\{A_i\}$ and $\bar \A=\{1,2,\dots,n\} \setminus (1+\A)$.
Similarly we define $d_i$ and $D_i=\sum_{j=1}^id_j$ based on the length of the descents. Then define 
$\D=\cup_{i=1}^{m-1}\{D_i\}$ and $\bar \D= \{1,2,\dots,n\} \setminus \D$.
We also set $A_0=D_0=0$.
Let $\tau_{\gamma}$ be the corresponding  \textbf{321}-avoiding permutation with the BJS bijection. This is defined by
$\tau_{\gamma}(D_i)=1+A_i$ on $\D$ and such that $\tau|_{\bar \D}=\bar \A$ 
and is increasing on $\bar \D$.

\begin{figure}
\centering
\begin{tikzpicture}

\draw [step=.5, help lines] (0,0) grid (10,4);

\draw (0,0)--(0, 0)--(1/2, 1/2)--(1, 0)--(3/2, 1/2)--(2, 1)--(5/2, 3/2)--(3, 2)--(7/2, 3/2)--(4, 1)--(9/2, 3/2)--(5, 2)--(11/2, 5/2)--(6, 3)--(13/2, 5/2)--(7, 2)--(15/2, 5/2)--(8, 2)--(17/2, 3/2)--(9, 1)--(19/2, 1/2)--(10, 0);

[place/.style={circle,draw,minimum size=6mm},
transition/.style={rectangle,draw,minimum size=4mm}]

\node at (3,2) [circle,draw]{};
\node at (.5,.5) [circle,draw]{};
\node at (6,3) [circle,draw]{};

\node at (1,0) [rectangle,draw] {};
\node at (7,2) [rectangle,draw] {};
\node at (4,1) [rectangle,draw] {};

\end{tikzpicture} 

\ \\

\begin{tikzpicture}
\draw [step=1,] (0,0) grid (11,2);

\node at (.5,.5) {$\tau(i)$};
\node at (.5,1.5) {$i$};
\foreach \i in {1,...,10}
	\node at (.5 + \i,1.5) {\i};	

\draw (1.25,1.25) rectangle (1.75,1.75);
\draw (3.25,1.25) rectangle (3.75,1.75);
\draw (5.25,1.25) rectangle (5.75,1.75);
\draw (1.5,.5) circle (.35cm);
\draw (3.5,.5) circle (.35cm);
\draw (5.5,.5) circle (.35cm);
\node at (1.5,.5) {2};
\node at (3.5,.5) {6};
\node at (5.5,.5) {10};

\node at (2.5,.5) {1};
\node at (4.5,.5) {3};
\node at (6.5,.5) {4};
\node at (7.5,.5) {5};
\node at (8.5,.5) {7};
\node at (9.5,.5) {8};
\node at (10.5,.5) {9};

\end{tikzpicture}
\caption{Dyck path of length 20 with correspond \text{bf 321}-avoiding permutation.}
\end{figure}
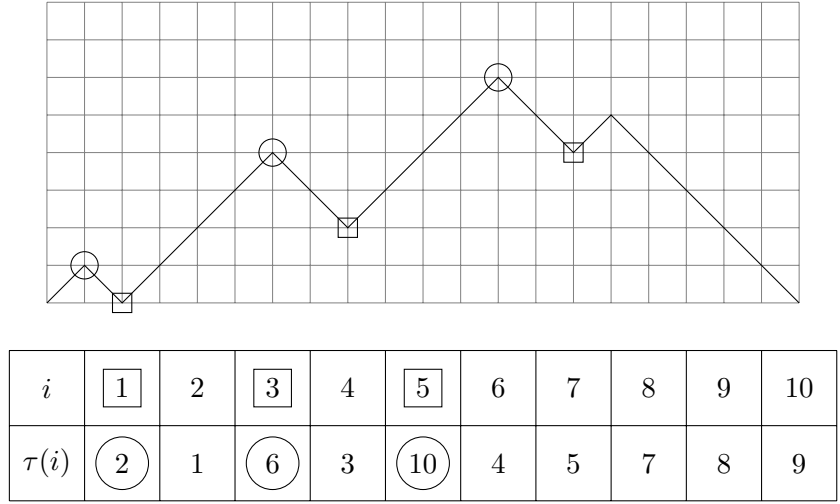

For the rest of the section we let $y_0=0$ and 
\begin{equation}
\label{neal}
y_i=A_i-D_i=\gamma(A_i+D_i).
\end{equation}

We will make use of the following property of this bijection whose proof we leave to the reader.
\begin{lemma} \label{downandout}
For any $\gamma \in \dn$ and $j \in \{1,2,\dots,n\}$
$$\tau_{\gamma}(j)>j \ \ \ \text{ if } j \in \D$$
and
$$\tau_{\gamma}(j) \leq j \ \ \ \text{ if } j \not\in \D$$
\end{lemma}

 We now undertake a more detailed analysis to show that for most \textbf{321}-avoiding permutations if $i,j$ are such that 
 $D_{i-1}<j< D_i$ then $$\tau_{\gamma}(j) \approx j-y_i$$ and if
 $j= D_i$ then $$\tau_{\gamma}(j) \approx j+y_i.$$
 Our first step is the following lemma.
\begin{lemma} \label{stanford}
Fix a Dyck path $\gamma \in \dn$ and $j \not \in \D$. There exists $i$ such that 
$D_{i-1}<j<D_i$. Then for $k \in\{1,\dots,m-1\}.$
\begin{enumerate}
\item If $A_k-(k-1)>D_i-i$ \hspace{.555in}then $\tau_{\gamma}(j)<A_k$.
\item If $A_k-(k-1)<D_{i-1}-(i-1)$ then $\tau_{\gamma}(j)>A_k+1$.
\end{enumerate}
\end{lemma}

\begin{proof}
Let $x=\max\bigg(\bar \A \cap\{1,2,\dots,A_k\}\bigg).$
In the first case we note that 
$$|\bar \A \cap\{1,2,\dots,A_k\}|=A_k-(k-1)>D_i-i=|\{1,2,\dots,D_i\}\cap \bar \D|.$$
Thus $$\tau_{\gamma}^{-1}(x)>D_i>j.$$
As $\tau_{\gamma}$ is monotone on the complement of $\D$ we get that 
$$A_k\geq x =\tau_{\gamma}(\tau_{\gamma}^{-1}(x))>\tau_{\gamma}(j).$$

In the second case
$$|\bar \A \cap\{1,2,\dots,A_{k}\}|=A_{k}-(k-1)<D_{i-1}-(i-1)
=|\{1,2,\dots,D_{i-1}\}\cap \bar \D|.$$
Thus $\tau_{\gamma}^{-1}(x)<D_{i-1}<j$ and
as $\tau_{\gamma}$ is monotone on the complement of $\D$
$$x = \tau_{\gamma}(\tau_{\gamma}^{-1}(x))<\tau_{\gamma}(j).$$
As $\tau_{\gamma}(j)>x$ and $\tau_{\gamma}(j) \in \bar \A$  we get that 
$$\tau_{\gamma}(j)>1+A_k.$$
\end{proof}

\begin{definition} \label{disparate}
We say that a Dyck path $\gamma \in \dn$ with associated sequences $A_i$ and $D_i$ satisfies the Petrov conditions if
\begin{enumerate}[(a)]
\item $\max_{x \in \{0,1,...,2n\}} \gamma(x)<.4n^{.6}$ \label{htcone}
\item $|\gamma(x)-\gamma(y)|<.5n^{.4}$ for all $x,y$ with $|x-y|<2n^{.6}$ \label{porter}

\item  $|A_i-A_j-2(i-j)|<.1|i-j|^{.6}$ for all $i,j$ with $|i-j|\geq n^{.3}$ and \label{dorsett}
\item  $|D_i-D_j-2(i-j)|<.1|i-j|^{.6}$ for all $i,j$ with $|i-j|\geq n^{.3}$ \label{tony}
\end{enumerate}
\end{definition}

\begin{lemma} \label{petrov}
With high probability the Petrov conditions are satisfied. The probability that they are not satisfied is decaying exponentially in $n^{c}$ for some $c>0$.
\end{lemma}
\begin{proof}
These results are standard Petrov style moderate deviation results \cite{Petrov1}. The general type of conditioning argument we need appears in \cite{MaMo03, PiRi13}. However we have not seen the exact results that we need anywhere in the literature so we include proofs of these statements in Appendix \ref{appendixb}.
\end{proof}

From these conditions we can derive many other moderate deviation results. We now list the ones that we will need.
The proofs of these lemmas are contained in \cite{part1}.

\begin{lemma} \label{voucher}
If a Dyck path $\gamma \in \dn$ with associated sequences $A_i$ and $D_i$ satisfies the Petrov conditions then $y_i <n^{.4}$ for all $i<n^{.6}$ and for all $i>n-n^{.6}$.  Also, $|A_i-A_{i-1}|, |D_i-D_{i-1}|<n^{.18}$ for all $i$.  This implies $|y_i-y_{i-1}|<n^{.18}$ for all $i$. 
Finally every consecutive sequence of length at least $n^{.3}$ has at least one element of $\D$ and at least one element of $\bar \D$.
\end{lemma}

\begin{lemma} \label{ducks}
For any Dyck path $\gamma \in \dn$ and any $j$ such that $D_{i-1}<j<D_i$ we get the following.
If the Petrov conditions are satisfied then

$$|\tau_{\gamma}(j)-j+y_i|<7n^{.4}$$
\end{lemma}

\begin{lemma} \label{fathersday}
For any Dyck path $\gamma \in \dn$ that satisfies the Petrov conditions and any $j=D_i \in \D$
$$|\tau_{\gamma}(j)-j-\gamma(2j)|<10n^{.4} .$$
Also for any such $\gamma$, $j$ and $i$ with $D_{i-1}<j<D_i$
$$|\tau_{\gamma}(j)-j+\gamma(2j)|<10n^{.4} .$$
\end{lemma}
%
%
%
\section{Fixed points of {\bf 123}-avoiding permutations}
In this section we use our analysis of {\bf 321}-avoiding permutations from Section \ref{offer} to study the fixed points of a random {\bf 123}-avoiding permutation. A permutation with three distinct fixed points has the pattern 123.
Thus a {\bf 123}-avoiding permutation can have at most 2 fixed points. At most one of them can be in the interval $[1,n/2]$ and at most one of them in the interval $(n/2,n]$. Elizalde  showed that as $n \to \infty$ the expected number of fixed points in a random {\bf 123}-avoiding permutation is converging to $1/2$ \cite{elizalde2004statistics}. Miner and Pak refined this by showing that the number of fixed points outside of the interval $[(1-\epsilon)n/2,(1+\epsilon)n/2]$ is converging to 0 \cite{mp}. In this section we give an asymptotic description of the distribution of fixed points in terms of Brownian excursion. 

We start with our main combinatorial lemma.
\begin{lemma} \label{pumpkin}
Let $\thingtwo$ be a {\bf 123}-avoiding permutation of length $n$, let 
$\thingone$ defined by $\thingone(k)=n+1-\thingtwo(k)$
and
et $\gamma \in \dyck{2n}$ be the image of $\thingone$ under the BJS bijection. There is a local minimum of $\gamma$ at $n$ if and only if there exists $k$ such that $p(k)=k \leq n/2.$ If there exists a fixed point at some $k \leq n/2$ then the fixed point $k$ satisfies
\begin{equation} \label{palestinian citizens}
k=\frac{n-\gamma(n)}2.\end{equation}
\end{lemma}

\begin{proof}
It is clear by symmetry that $\thingtwo$ is {\bf 123}-avoiding if and only if $\thingone$ is {\bf 321}-avoiding. We note that there is a fixed point $k=\thingtwo(k)$ if and only if there is a $k$ such that  
$(k, \thingone(k))$ is on the anti-diagonal of the graph of $\thingone$, i.e.\ $\thingone (k)=n+1-k.$

If $k\leq n/2$ and $\thingtwo(k)=k$ then we have $\thingone(k)=n+1-k>k$ and $(k,\thingone(k))$ 
lies above the diagonal and on the upper sequence. Thus we must have that $k=D_i=D_i(\gamma)$ 
for some $i$ and $$\thingone (k)=1+A_i=n+1-D_i.$$ 
Rearranging we get that
$A_i+D_i=n.$ This implies that $$(A_i+D_i,A_i-D_i)=(n,A_i-D_i)$$ is a local minimum on the graph of  $\gamma$. 

Similarly if $(n,n-2j)$ is a local minimum of $\gamma$ then there exists $k$ such that 
$A_k+D_k=n.$ Then 
$$\thingone (D_k) = 1+A_k=1+n-D_k$$
and $(D_k,\thingone(D_k))$ lies on the anti-diagonal. As $D_k \leq A_k=n-D_k$ we get $D_k \leq n/2$. Solving $A_i+D_i=n$ and $A_i-D_i=\gamma(n)$ for $D_i$ we get that 
$$D_i=\frac{n-\gamma(n)}{2}$$
which is the location of the fixed point.
\end{proof}

\begin{figure}
\centering
\begin{tikzpicture}
\draw[step=.3, help lines] (0,0) grid (7.21,3); 
\draw[step=.3, help lines] (7.49,0) grid (7.5 + 7.21,3) ;

\draw (0,0)--
(0, 0)--
(3/10, 3/10)--
(3/5, 3/5)--
(9/10, 9/10)--
(6/5, 6/5)--
(3/2, 9/10)--
(9/5, 3/5)--
(21/10, 9/10)--
(12/5, 6/5)--
(27/10, 3/2)--
(3, 9/5)--
(33/10, 3/2)--
(18/5, 6/5)--
(39/10, 3/2)--
(21/5, 6/5)--
(9/2, 3/2)--
(24/5, 9/5)--
(51/10, 3/2)--
(27/5, 9/5)--
(57/10, 3/2)--
(6, 6/5)--
(63/10, 9/10)--
(33/5, 3/5)--
(69/10, 3/10)--
(36/5, 0);

\draw[ultra thick] 
(33/10, 3/2)--
(18/5, 6/5)--
(39/10, 3/2);
\draw 
(15/2, 0)--
(39/5, 3/10)--
(81/10, 3/5)--
(42/5, 9/10)--
(87/10, 6/5)--
(9, 9/10)--
(93/10, 3/5)--
(48/5, 9/10)--
(99/10, 6/5)--
(51/5, 3/2)--
(21/2, 9/5)--
(54/5, 3/2)--
(111/10, 9/5)--
(57/5, 3/2)--
(117/10, 6/5)--
(12, 3/2)--
(123/10, 9/5)--
(63/5, 3/2)--
(129/10, 9/5)--
(66/5, 3/2)--
(27/2, 6/5)--
(69/5, 9/10)--
(141/10, 3/5)--
(72/5, 3/10)--
(147/10, 0);

\draw[ultra thick]
(54/5, 3/2)--
(111/10, 9/5)--
(57/5, 3/2);

\draw (18/5, 6/5) -- (18/5,0);

\draw (18/5+.75, .5) node {$h=4$};
\draw (18/5, -.25) node {$n$};
\draw (111/10, -.25) node {$n$};
\draw (.04, -.25) node {$0$};
\draw (7.54, -.25) node {$0$};
\draw (35/5, -.25) node {$2n$};
\draw (7.5+35/5, -.25) node {$2n$};

\draw (6.71,1.7) node {$\gamma$}; 
\draw(14.22,1.7) node {$\gamma'$} ;

\draw[step=1/2, help lines] (.49,-7) grid (6.51,-1); 
\draw[step=1/2, help lines] (7.99,-7) grid (14,-1);

\foreach \i in {0,...,5}
{
	\draw[ultra thick] (\i/2+.5,-\i/2-1)--(\i/2 + .5+.5,-\i/2-1)--(\i/2 + .5+.5,-\i/2-1-.5)--(\i/2 + .5,-\i/2-1-.5)--cycle;
	\draw[ultra thick] (7.5+ \i/2+.5,-\i/2-1)--(7.5+ \i/2 + .5+.5,-\i/2-1)--(7.5+\i/2 + .5+.5,-\i/2-1-.5)--(7.5+ \i/2 + .5,-\i/2-1-.5)--cycle;
}

\draw (3.5, -7.5) node {$\tau_\gamma$};
\draw (3+8, -7.5) node {$\tau_{\gamma'}$};

\draw[fill] (.25 + .5, .25-1  + -6 ) circle (5pt);
\draw[fill] (.25 + .5*2, .25-1  + -6 + 4*.5 ) circle (5pt);
\draw[fill] (.25 + .5*3, .25-1  + -6 +.5) circle (5pt);
\draw[fill] (.25 + .5*4, .25-1  + -6 +8*.5) circle (5pt);
\draw[fill] (.25 + .5*5, .25-1  + -6 +9*.5) circle (5pt);
\draw[fill] (.25 + .5*6, .25-1  + -6 +11*.5) circle (5pt);
\draw[fill] (.25 + .5*7, .25-1  + -6 +2*.5) circle (5pt);
\draw[fill] (.25 + .5*8, .25-1  + -6 +3*.5) circle (5pt);
\draw[fill] (.25 + .5*9, .25-1  + -6 +5*.5) circle (5pt);
\draw[fill] (.25 + .5*10, .25-1  + -6 +6*.5) circle (5pt);
\draw[fill] (.25 + .5*11, .25-1  + -6 +7*.5) circle (5pt);
\draw[fill] (.25 + .5*12, .25-1  + -6 +10*.5) circle (5pt);

\draw[fill] (.25 + .5+7.5, .25-1  + -6 ) circle (5pt);
\draw[fill] (.25 + .5*2+7.5, .25-1  + -6 + 4*.5 ) circle (5pt);
\draw[fill] (.25 + .5*3+7.5, .25-1  + -6 +8*.5) circle (5pt);
\draw[fill] (.25 + .5*4+7.5, .25-1  + -6 +.5) circle (5pt);
\draw[fill] (.25 + .5*5+7.5, .25-1  + -6 +9*.5) circle (5pt);
\draw[fill] (.25 + .5*6+7.5, .25-1  + -6 +11*.5) circle (5pt);
\draw[fill] (.25 + .5*7+7.5, .25-1  + -6 +2*.5) circle (5pt);
\draw[fill] (.25 + .5*8+7.5, .25-1  + -6 +3*.5) circle (5pt);
\draw[fill] (.25 + .5*9+7.5, .25-1  + -6 +5*.5) circle (5pt);
\draw[fill] (.25 + .5*10+7.5, .25-1  + -6 +6*.5) circle (5pt);
\draw[fill] (.25 + .5*11+7.5, .25-1  + -6 +7*.5) circle (5pt);
\draw[fill] (.25 + .5*12+7.5, .25-1  + -6 +10*.5) circle (5pt);

\end{tikzpicture}

\caption{ The pictures above show two Dyck paths $\gamma,\gamma'\in \text{Dyck}^{2n}$ and corresponding {\bf321}-avoiding permutations $\tau_\gamma,\tau_{\gamma'}$.  The path $\gamma$ has a local minimum at $(n,\gamma(n))$.  This corresponds with the point $\left(\frac{n- \gamma(n)}{2},\frac{n+\gamma(n)}{2}+1\right)$ on the anti-diagonal in $\tau_\gamma$.  This becomes a fixed point for $\rho_\gamma=n+1-\tau_\gamma$.  The other path $\gamma'$ does not have a local minimum at $n$ and correspondingly $\tau_{\gamma'}$ has no point on the anti-diagonal with $x$-coordinate less than or equal to $n/2$ and no fixed point of $\rho$ less than or equal to $n/2$.}
\end{figure}
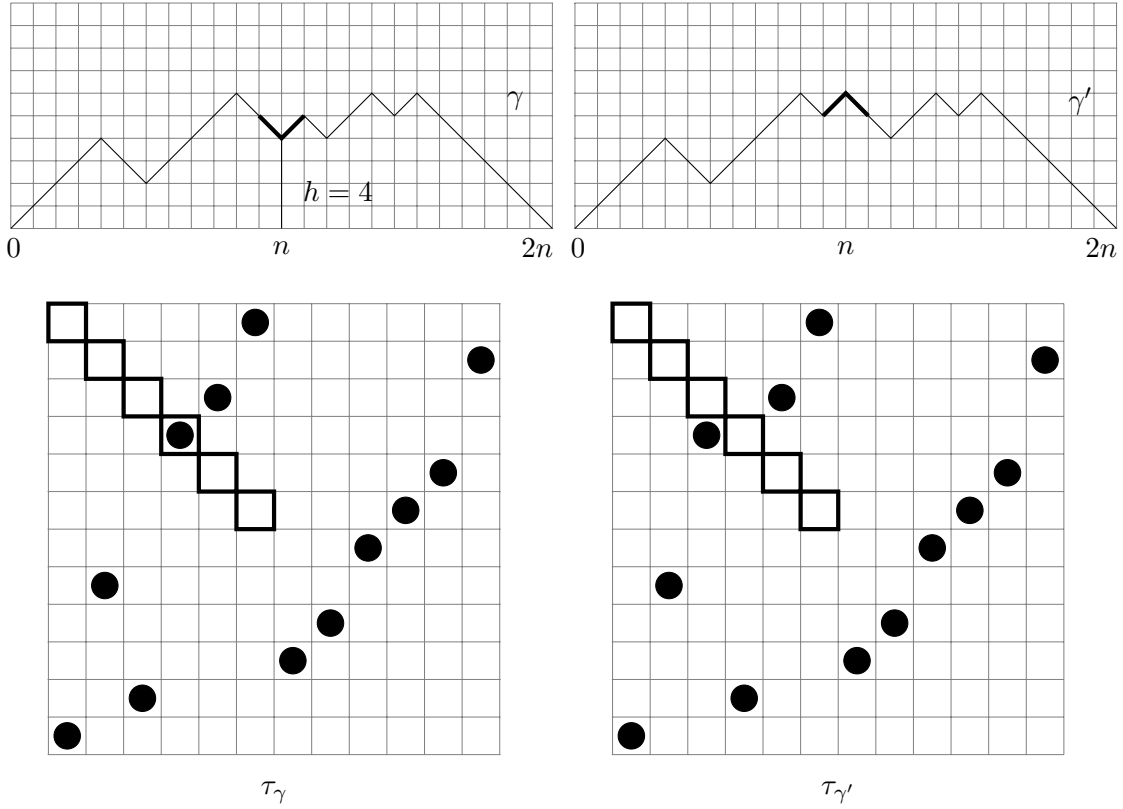

We now translate Lemma \ref{pumpkin} into a statement about the distribution of fixed points.
To perform this analysis we define several random variables on the set of {\bf 123}-avoiding permutations. 
\begin{itemize} 
\item For any $ x\in \R \cup \infty$ let $\delta_x$ be the point mass at $x$. 
\item $\tilde A_n(\thingtwo)=\#\{i \in [1,n/2]:\ \thingtwo(i)=i\}$
\item $\tilde B_n(\thingtwo)=\#\{i \in (n/2,n]:\ \thingtwo(i)=i\}$
\item $\hat X_n(\thingtwo)$ to be the fixed point in $[1,n/2]$ if it exists and $\infty$ if there are none.
\item $\hat Y_n(\thingtwo)$ to be the fixed point in $(n/2,n]$ if it exists and $\infty$ if there are none.
\end{itemize}

From these two we define the random measures 
 $$\tilde X_n(\thingtwo) =\tilde A_n(\thingtwo)\delta_{\frac{\hat X_n-n/2}{\sqrt{2n}}}
 =\sum_{i=1}^{\lfloor n/2\rfloor}\delta_{\frac{i-\frac{n}2}{\sqrt{2n}}} \ind_{\thingtwo(i)=i}$$

and
$$ \tilde Y_n(\thingtwo)= \tilde B_n(\thingtwo) \delta_{\frac{\hat Y_n-n/2}{\sqrt{2n}}}
=\sum_{i=\lfloor n/2\rfloor+1}^n\delta_{\frac{i-\frac{n}2}{\sqrt{2n}}} \ind_{\thingtwo(i)=i}.$$
So $\tilde X_n(\thingtwo)+ \tilde Y_n(\thingtwo)$ encodes the number and location of the fixed points and is appropriately scaled. For the above random variables we often drop the $\thingtwo$ when we are referring to uniformly random $\textbf{123}$-avoiding permutation.

Now we identify the limit of  $\tilde X_n +\tilde Y_n$ which is the main result of this section.

\begin{theorem} \label{ale}
Let $A$ and $B$ be independent Bernoulli(1/4) random variables and let $X$  be a random variable that is independent of $A$ and $B$ and distributed like $\frac12\exc_{1/2}$, half the height of a Brownian excursion at $1/2$. Then
$$ \tilde X_n+ \tilde Y_n \xrightarrow{\text{dist}} A \delta_{-X}+B \delta_X. $$
\end{theorem}

Note that this is an equivalent formulation of Theorem \ref{intro main} Part (b). The first step in the proof is the following lemma.

\begin{lemma} \label{vonnegut said}
$$\tilde X_n \xrightarrow{\text{dist}} A \delta_{-X} \qquad \text{and} \qquad
\tilde Y_n \xrightarrow{\text{dist}} B \delta_X.$$
\end{lemma}

\begin{proof}
Let $\thingtwo$, $\thingone$ and $\gamma$ be as in Lemma \ref{pumpkin}.
If $\thingtwo$ is a {\bf 123}-avoiding permutation then $\thingone$ 
is {\bf 321}-avoiding. 
By Lemma \ref{pumpkin} we have that if $D_i\leq n/2$ is a fixed point for $\thingtwo$ then $(n,A_i-D_i)$ is a local minimum for $\gamma$. The set of Dyck paths of length $2n$ that have a local minimum of $(n,h)$ is in 1-1 correspondence with the set of Dyck paths of length $2n-2$ that go through $(n-1,h+1).$ 
If $n$ is even we get
\begin{equation} \label{tatiana}
\P(\tilde A_n=1)=\frac{C_{n-1}}{C_n} \to \frac14.
\end{equation}
If $n$ is odd we get
\begin{equation} \label{toro}
\P(\tilde A_n=1)=\frac{C_{n-1}-C_{\frac{n-1}{2}}C_{\frac{n-1}{2}}}{C_n} \to \frac14.
\end{equation}

Because of \eqref{palestinian citizens} in Lemma \ref{pumpkin} we get that  
$$\frac{\hat X_n -n/2}{\sqrt{2n}} \ \textrm{ given } (\tilde A=1) \overset{dist}{=\joinrel=} \frac{-\Gamma^{n}(n)}{2\sqrt{2n}} \ \
\text{ given $(n,\Gamma^n(n))$ is a local minimum,}$$ which is equal in distribution (if $n$ is even) to
$$\frac{-\Gamma^{n-1}(n-1)+1}{2\sqrt{2n}} $$
where $\Gamma^{n-1}$ is a uniformly chosen Dyck path of length $2n-2$. This last quantity is 
converging in distribution to $-X=-\mathbbm{e}_{1/2}/2$. The case when $n$ is odd is virtually identical. This proves the first claim in the lemma.

The second follows by symmetry as the permutation defined by
$$\bar \thingtwo(k)=n+1-\thingtwo(n+1-k)$$
is also {\bf 123}-avoiding. The fixed points of $\bar \thingtwo$ are $n+1$ minus the fixed points of $\thingtwo$. Thus the convergence of $\tilde Y_n \to B\delta_X$ follows by symmetry.
\end{proof}

Our next goal is to prove the following lemma.
\begin{lemma} \label{we are put on earth}
$$(\tilde A_n,\tilde B_n) \xrightarrow{\text{dist}} (A,B).$$
\end{lemma}

To prove that $\tilde A_n$ and $\tilde B_n$ are asymptotically independent consider $\gamma \in \dyck{2n}$ and its image under the BJS bijection $\thingtwo_\gamma$. As we saw in Lemma \ref{pumpkin} $\tilde A_n=1$ is the event
$$\gamma(n-1)-\gamma(n)=\gamma(n+1)-\gamma(n)=1$$
which is determined by the increments of $\gamma$ in the region $[n-1,n+1].$
We will now show that $\tilde B_n$ is essentially determined by $\gamma$ in 
$$[0,2n]\setminus [n-1,n+n^{.4}].$$
As the increments of a Dyck path in $[n-1,n+1]$ are roughly independent of the values of $\gamma$ in  $[0,2n]\setminus [n-1,n+n^{.4}]$ we will get that $\tilde A_n$ is asymptotically independent of $\tilde B_n$.

To make this formal we define an equivalence relation on $\dyck{2n}$. 
\begin{definition}
 For $\gamma,\gamma' \in \dyck{2n}$ we write $\gamma \sim \gamma'$ if 
\begin{itemize}
\item $\gamma(m)=\gamma'(m)$ for all $m \not \in [n,n+l-2]$ where $l=\lfloor n^{.4}\rfloor$ and
\item $\gamma$ and $\gamma'$ have the same number of local miximums and local minimums in the interval
$$[n-2,n+l].$$
\end{itemize}
Let $S$ denote the set of all equivalence classes for $\sim$. 
\end{definition}

We now define a good set of equivalence classes. Then we show that almost all the Dyck paths are in their union, which we call $G_{n}$.
\begin{definition}
Define ${\mathcal G}_{n}$ to be the set of all equivalence classes $s \in S$ such that 
\begin{enumerate}
\item some element $\gamma \in s$ satisfies the Petrov conditions (defined in Definition \ref{disparate}) and
\item  $\gamma(n-2)>n^{.45}$ for some (all) $\gamma \in s.$
\end{enumerate}
Also define 
\begin{equation} \label{divestment}
G_{n}=\bigcup_{s \in \mathcal G_{n}} s
\end{equation}
\end{definition}

\begin{lemma}\label{cauliflower}
$\P(G_{n}) \to 1 \text{ as } n \to \infty.$
\end{lemma}
\begin{proof}
 The probability that the Petrov conditions are not satisfied is decaying exponentially in $n^c$ for some $c > 0$ by Lemma \ref{petrov}.
The second condition in the definition of ${\mathcal G}_n$ is true for all but a set of $\gamma$ of order $O(n^{-c})$ for some $c>0$.
\end{proof}

\begin{lemma} \label{moyes}
$$ \lim_{n \to \infty} \max_{s \in {\mathcal G}_{n}} \left| \E( \text{ $\tilde A_n$ } |\  s)- \frac14 \right| \to 0 .$$

\end{lemma}
\begin{proof}
This is a straightforward but tedious calculation. 
Fix $a,b \in \{\pm 1\}$ and $h, h', j \in \N$. Let $s$ be an equivalence class such that 
$$\gamma(n-2)=h+a,\text{ and }  \gamma(n-1)=h, \quad \text{ and } \quad \gamma(n+l-1)=h'+b \ 
\text{ and }\gamma(n+l)=h'$$ 
and there are $j$ peaks in the interval $[n-2,n+l]$. We break $s$ up into four sets based on whether $\gamma(n-1)-\gamma(n)$ and $\gamma(n)-\gamma(n+1)$ are positive or negative. (Note that the set where $\gamma$ has a local minimum  at $n$ is one of those four sets.) 
The cardinality of these four sets are explicitly calculated in Proposition 9 of \cite{labarbe}. It is easy to show that if some element of $s$ satisfies the Petrov conditions then the ratio of the sizes of any of these sets is $1+o(1).$  We leave the details to the reader. 
\end{proof}

\begin{lemma} \label{tildeB}
For any $\gamma, \gamma' \in G_{n}$ with $\gamma \sim \gamma'$ we have 
\begin{equation*} \label{yams}
\tilde B_n(\thingtwo_\gamma)=\tilde B_n(\thingtwo_{\gamma'})
\qquad \text{and} \qquad
\hat Y_n(\thingtwo_\gamma)=\hat Y_n(\thingtwo_{\gamma'})
\end{equation*}
\end{lemma}
\begin{proof}
Remember the definitions of $\D, \A$ and $\bar \A$ from the start of Section \ref{offer}.
First we claim that
\begin{equation} \label{polenta}
\D_\gamma \Delta \D_{\gamma'}\subset  [1,n/2] \qquad \text{and} \qquad \bar \A_\gamma \Delta \bar \A_{\gamma'} \subset [1+n/2,n].
\end{equation}
To see this note that both of these sets are defined by the points which are a local minimum for one Dyck path but not the other. Based on the definition of the equivalence relation the local minima of 
$\gamma$ and $\gamma'$ can only differ in the interval $I=(n-2,n+n^{.4}).$ By the second condition in the definition of ${\mathcal G}_n$ each of the local minima in the interval $I$ is preceded by at least $n/2+n^{.45}/2-2>1+n/2$ up-steps. This proves the second claim in \eqref{polenta}. Also by the second condition in the definition of ${\mathcal G}_n$ each of the local minima in the interval $I$ is preceded by at most $n/2-n^{.45}/2+n^{.4}<n/2$ down-steps. This proves the first claim in \eqref{polenta}. Also by the previous argument and the second condition in the equivalence relation
\begin{equation}\label{the new class}
|\D_\gamma\cap [1,n/2]|= | \D_{\gamma'}\cap  [1,n/2]|.
\end{equation}

If $\tilde B_n(\thingtwo(\gamma))=1$ then there exists $j> n/2$ which is a fixed point of $\thingtwo_\gamma$ and it lies on the anti-diagonal of $\thingone_\gamma$. As $j> n/2$ we get 
\begin{equation} \label{bhg}
\thingone_\gamma(j)=n+1-j \leq j \qquad \text{and} \qquad \thingone_\gamma(j)=n+1-j < 1+n/2.
\end{equation}
So by the first part of \eqref{bhg} we have $(j,\thingone_\gamma(j))$ lies on the lower sequence for $\thingone_\gamma$. Thus $j \not \in \D_\gamma$. By the first part of \eqref{polenta} and the fact that $j>n/2$ we also have $j \not \in \D_{\gamma'}$. Thus $\thingone_{\gamma}(j) \in \bar \A_\gamma$ and $\thingone_{\gamma'}(j) \in \bar \A_{\gamma'}$.
By \eqref{the new class} and the fact that $\D_{\gamma}$ and $\D_{\gamma' }$ are equal after $n/2$ there exists $k$ such that $$D_k=D'_k <j<D_{k+1}=D'_{k+1}.$$ So $\thingone_\gamma(j)$ is the $j-k$th element of $\bar \A_\gamma$ and $\thingone_{\gamma'}(j)$ is the $j-k$th element of $\bar \A_{\gamma'}.$  By the second half of \eqref{polenta} and the second part of \eqref{bhg} we know that  
$$\thingone_\gamma(j)\in \bar \A_\gamma|_{[1,1+n/2)}=\bar \A_{\gamma'}|_{[1,1+n/2)}.$$
Thus $ \thingone_\gamma(j)$ must be equal to $\thingone_{\gamma'}(j)$ as they are both the $j-k$th term in the same set.
Thus $(j,\thingone_\gamma(j))$ and$(j,\thingone_{\gamma'}(j))$ lie on the anti-diagonal and $j=\thingtwo_{\gamma'}(j)$ is a fixed point of $\thingtwo_{\gamma'}$. As the roles of $\gamma$ and $\gamma'$ are symmetric this establishes the claim of the lemma.
\end{proof}

\begin{pfoflem}{\ref{we are put on earth}}
From \eqref{tatiana} and \eqref{toro} we have that $\tilde A_n \to A$ and by symmetry we have that $\tilde B_n \to B$. Thus we just need to show that $\E(\tilde A_n \tilde B_n) \to \frac1{16}.$ 
\begin{eqnarray}
\left| \E(\tilde A_n \tilde B_n) -\frac{1}{16}\right| 
& = & \left| \E(\tilde B_n) \E(\tilde A_n|\tilde B_n=1) -\frac{1}{16}\right| \nonumber\\ 
& \leq & \left| \E(\tilde B_n)  -\frac{1}{4}\right| +  \left|  \E(\tilde A_n|\tilde B_n=1)-\frac14 \right| \nonumber\\
& \leq & \left| \E(\tilde B_n)  -\frac{1}{4}\right| +  \left|  \E(\tilde A_n|\tilde B_n=1) -\E(\tilde A_n) \right| +\left|\E(\tilde A_n)-\frac14 \right| \nonumber\\
& \leq & \left| \E(\tilde B_n)  -\frac{1}{4}\right| +  \left|  \E(\tilde A_n|\{\tilde B_n=1\}\cap G_{n})-\E(\tilde A_n) \right|+\P\left(G_{n}^C\right) +\left|\E(\tilde A_n)-\frac14 \right| \nonumber\\
& \leq & \left| \E(\tilde B_n)  -\frac{1}{4}\right| + \max_{s \in {\mathcal G}_{n}} \left| \E(\tilde A_n)- \E(\tilde A_n| s) \right|+\P\left(G_{n}^C\right)+\left|\E(\tilde A_n)-\frac14 \right| . \label{finalline}
\end{eqnarray}
The last inequality  is valid because of Lemma  \ref{tildeB}. The first and last terms on the right hand side of \eqref{finalline} go to zero by Lemma  \ref{we are put on earth}.  
The second term goes to zero by Lemmas \ref{we are put on earth} and  \ref{moyes} and the third term goes to zero by Lemma \ref{cauliflower}.
\end{pfoflem}

In contrast to Lemma \ref{pumpkin} the event $\{\tilde B_n=1\}$ and the location $\hat Y_n$ of the fixed point after $n/2$ is more complicated to describe.

\begin{lemma} \label{library}
For all $n$ sufficiently large and all $\gamma \in G_{n}$ with $\tilde B_n=1$
$$|\hat Y_n-n/2-\gamma(n)/2|\leq 100n^{.4} $$
\end{lemma} 

\begin{proof}
Suppose $\gamma \in G_n$. By the definition of $G_n$ there exists $\gamma'$ such that $\gamma \sim \gamma'$ and $\gamma'$ satisfies the Petrov conditions. By Lemma \ref{tildeB} it causes no loss of generality to assume that $\gamma$ satisfies the Petrov conditions. 

Restricted to the lower sequence $\bar \D$ we have that $\tau_\gamma(j)+j$ is an increasing sequence as each component is increasing.
We will show that if $j  \in \bar \D$ and 
\begin{equation} \label{steel plate} j< \frac{n+\gamma(n)-100n^{.4}}{2} \end{equation}
then $\tau_\gamma(j)+j<n.$ 
Similarly we will show that if $j  \in \bar \D$ and 
\begin{equation}
j> \frac{n+\gamma(n)+100n^{.4}}{2} \label{electric bugaloo} \end{equation}
then $\tau_\gamma(j)+j>n+1.$ 
Then for any $j$ with $(j, \tau_\gamma(j))$ on the anti-diagonal and the lower sequence we must have $j+\tau_\gamma(j)=n+1$ and thus 
$$j \in \left(\frac{n+\gamma(n)-100n^{.4}}{2},\frac{n+\gamma(n)+100n^{.4}}{2}\right).$$
As the probability we are considering a set of $\gamma$ of almost full probability this is sufficient to prove the lemma.

Let $j$ be the smallest value in $\bar \D$ not satisfying \eqref{steel plate}. Then by Lemma \ref{voucher}
\begin{equation*} \label{steel plate star} j< \frac{n+\gamma(n)-98n^{.4}}{2} \end{equation*}
 Since $\gamma$ satisfies the Petrov conditions by Lemma \ref{fathersday}
we have 
$$\tau_\gamma(j)-j+\gamma(2j) \leq 10n^{.4}. $$
Then manipulating this we get
\begin{eqnarray}
\tau_\gamma(j)+j 
&\leq& 2j+10n^{.4}-\gamma(2j)  \nonumber\\
&\leq&n+\gamma(n)-98n^{.4}+10n^{.4}-
                            \gamma(n)+(\gamma(n)-\gamma(2j)) \nonumber\\
&\leq&n-88n^{.4}+n^{.4} \label{mason}\\
&<&n. \nonumber
\end{eqnarray}
The inequality in \eqref{mason} is true because by Petrov condition \eqref{htcone} 
 $n$ and $2j$ are within $n^{.6}$ and thus by Petrov condition \eqref{porter} $\gamma(n)$ and $\gamma(2j)$ are within
 $n^{.4}$. 

Let $j$ be the largest value in $\bar \D$ not satisfying \eqref{electric bugaloo}. Then by Lemma \ref{voucher}
\begin{equation*} \label{breaking two}j> \frac{n+\gamma(n)+98n^{.4}}{2} \end{equation*}
Since $\gamma$ satisfies the Petrov conditions by Lemma \ref{fathersday}
$$ \tau_\gamma(j)-j+\gamma(2j) \geq -10n^{.4}$$
Then manipulating and making the same estimates  we get
\begin{eqnarray*}
 \tau_\gamma(j)+j
 &\geq & 2j-10n^{.4} -\gamma(2j) \\
 &\geq & n+\gamma(n)+98n^{.4}-10n^{.4} -\gamma(n)+(\gamma(n)-\gamma(2j))\\
 &\geq & n+88n^{.4} +(\gamma(n)-\gamma(2j))\\
 &\geq & n+80n^{.4}.
 \end{eqnarray*}
The last line follows in the same way as the first computation.\end{proof}

\begin{lemma} \label{naan}
Conditional on $\tilde A_n=\tilde B_n=1$ we have that
$$ \tilde Y_n \xrightarrow{\text{dist}}  \delta_{X}.$$
\end{lemma}

\begin{proof}
Since $X$ is continuous using Lemma \ref{we are put on earth} suffices to show that for any $0<a<b$
\begin{equation} \label{laliga}
\P\left(\tilde A_n =\tilde B_n =1 \text{ and } \hat Y_n \in \left(a\sqrt{2n},b\sqrt{2n}\right)\right) \xrightarrow{\text{}} \frac{1}{16}\P(X \in (a,b)).\end{equation}
By Lemma \ref{vonnegut said} 
\begin{equation} \label{curry}
\P\left(\tilde B_n =1 \text{ and } \hat Y_n \in \left(a\sqrt{2n},b\sqrt{2n}\right)\right) \xrightarrow{\text{}} \frac{1}{4}\P(X \in (a,b)).\end{equation}
By Lemma \ref{tildeB} for any $s \in {\mathcal G}_n$  the event on the left hand side of \eqref{curry} either happens for all $\gamma \in s$ or for no $\gamma \in s.$   To calculate
$$\P\left(\tilde A_n=\tilde B_n =1 \text{ and } \hat Y_n \in \left(a\sqrt{2n},b\sqrt{2n}\right)\right)$$
we proceed as in the proof of Lemma \ref{we are put on earth}. By  Lemma \ref{moyes} we have that for each $s \in {\mathcal G}_{n}$
\begin{multline} 
\label{pedantic}\P\left(s \text{ and } \tilde A_n =  \tilde B_n =1 \text{ and } \hat Y_n \in \left(a\sqrt{2n},b\sqrt{2n}\right)\right)\\= \frac{1}{4}\P\left(s \cap \tilde B_n =1 \text{ and } \hat Y_n \in \left(a\sqrt{2n},b\sqrt{2n}\right)\right)(1+\Delta(s))
\end{multline}
where $\sup_{s \in  {\mathcal G}_{n}}\Delta(s)=o(1).$ 
Then we sum the terms on the left hand side over all $s \in S.$ By Lemma \ref{cauliflower} 
we have that $\P(G_n)$ is close to 1.
This completes the proof.
\end{proof}

\begin{lemma} \label{to fart around}
Conditional on $\tilde A_n=\tilde B_n=1$ we have that
$$ \tilde X_n+ \tilde Y_n \xrightarrow{\text{dist}}  \delta_{-X}+ \delta_X.$$
\end{lemma}

\begin{pfoflem}{\ref{to fart around}}
By Lemma \ref{pumpkin} for any $\gamma \in G_{n}$ with $\tilde A(\gamma)=1$ we have that
$$ \hat X(\gamma) =\frac{n-\gamma(n)}{2}.$$
Similarly by Lemma \ref{library} we have that
for any $\gamma \in G_{n}$ with $\tilde B_n(\gamma)=1$  
$$ \left|-\frac{n+\gamma(n)}{2} + \hat Y(\gamma) \right| \leq 100n^{.4}.$$
Combining these two lines gives us
\begin{equation*}
\left| \frac{\hat X- \frac{n}{2}}{\sqrt{2n}}+\frac{\hat Y-\frac{n}2}{\sqrt{2n}}\right| =
 \left|-\frac{\gamma(n)}{2\sqrt{2n}}  +\frac{\hat Y-\frac{n}2-\frac{\gamma(n)}2}{\sqrt{2n}} +\frac{\gamma(n)}{2\sqrt{2n}}\right|\leq \frac{100n^{.4}}{\sqrt{2n}}\leq 100n^{-.1}.
\end{equation*}

Combining this  with Lemma \ref{naan} establish the lemma.
\end{pfoflem}

\begin{pfofthm}{\ref{ale}}
Because of Lemma \ref{we are put on earth} it is sufficient to show that 
\begin{itemize}
\item conditional on $\tilde A_n=\tilde B_n=1$ then $\qquad \quad \tilde A_n \tilde X_n + \tilde B_n \tilde Y_n \to  \delta_{-X}+ \delta_X$,
\item conditional on $\tilde A_n=1-\tilde B_n=1$ then $\ \quad \tilde A_n \tilde X_n + \tilde B_n \tilde Y_n \to \delta_{-X}$ and
\item conditional on $1-\tilde A_n=\tilde B_n=1$ then $\ \quad \tilde A_n \tilde X_n + \tilde B_n \tilde Y_n \to \delta_X$.
\end{itemize}
The first of these is the statement of Lemma \ref{to fart around}.
Combining Lemma \ref{to fart around} with Lemma \ref{vonnegut said} shows that the second and third statements are true. This completes the proof.
\end{pfofthm}

%
%
%

\section{A Bijection between Dyck Paths and 231-avoiding permutations} \label{dyckpathsection}

The total number of Dyck paths from $0$ to $2n$ is given by $C_n$, the $n$th Catalan number.  The number of \textbf{231}-avoiding permutations in $S_n$ is also given by the $n$th Catalan number.  Hence there is a bijection between the two sets.  We now define a particular bijection that uses geometric properties of the path. Although we suspect this bijection does exist in the literature we are not sure where it does. For the sake of completeness we include a proof that it is a bijection here.
For our purposes the most important geometric aspect of a Dyck path is an excursion.
\begin{definition}
An {\bf excursion} in a Dyck Path starting at $x$ with height $h$ and length $l$ is a path interval $\gamma([x,x+l])$ such that
\begin{enumerate}
\item $\gamma(x)=\gamma(x+l)=h-1$ 
\item $\gamma(x+1)=\gamma(x+l-1)=h$ and
\item $l=\min\{j\geq1:\ \gamma(x+j)=h-1\}.$
\end{enumerate}
\end{definition}

Note that there are $n$ excursions in a Dyck Path of length $2n$ as there is one excursion that begins with every up-step.  Based on this correspondence we say the $i$th excursion, $Exc(i)$ is the one that begins with the $i$th up-step.

\begin{figure}
\centering

 \begin{tikzpicture} 
 \draw[step=.5, help lines] (0,0) grid (10,4);
 \draw (0,0) --(1.5,1.5);
 \draw (1.5,1.5)--(2,1);
 \draw (2,1)--(2.5,1.5)--(3,2)--(3.5,1.5) --(5,3)--(6,2)--(6.5,2.5)--(8.5,.5)--(9,1)--(10,0);
 \draw[ultra thick] (3.5,1.5) --(5,3)--(6,2)--(6.5,2.5)--(7.5,1.5);
 \draw[ultra thick] (4,0)--(4,2);
 \draw[ultra thick] (3.5,1.5) -- (7.5,1.5);
 \draw (9,1)--(10,0);
\end{tikzpicture}
\caption{A Dyck path in $\mathcal{D}_{10}$ with $v_6 = 8$, $h_6 = 4$, and $l_6 = 8$.}
 \label{fig.dyckex1}
\end{figure}
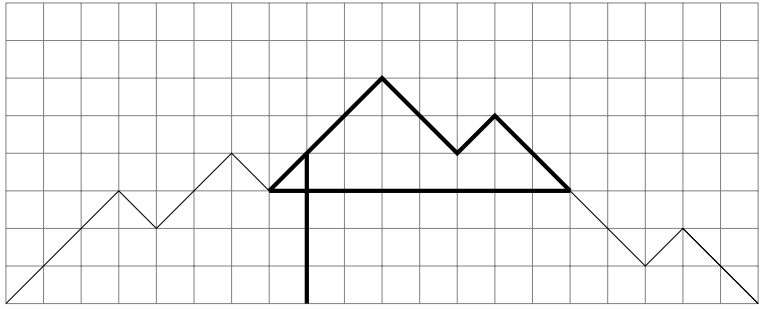

\begin{definition} \label{def.xihili}
For a Dyck path $\gamma$, define the following:

\begin{itemize}
\item $Exc(i):=$ the $i$th excursion.
\item $v_i:=$ the position after the $i$th up-step, or 1 + the start of $Exc(i).$
\item $h_i:=\gamma(v_i) = $ the height of the path after the start of $Exc(i).$
\item $l_i:=$ the length of the same excursion.  
\end{itemize}
\end{definition}
Figure \ref{fig.dyckex1} illustrates these definitions for a particular $\gamma.$

For a path $\gamma\in \dn$ we define pointwise the map $\sigma_\gamma = \sigma: [n] \to \Z$ by 
\begin{equation}\label{231 bijection}\sigma(i) = i + l_i/2  - h_i.
\end{equation}

\begin{theorem} \label{bijection} For $\gamma \in \dn$ let $\sigma = \sigma_\gamma$ be defined as above.  Then $\sigma\in \sn(231).$  Moreover, $\gamma\mapsto \sigma_\gamma$ is a bijection from $\dn \to \sn(231).$
\end{theorem}

\begin{proof}

A detailed version of this proof can be found in \cite{part1}.  We include a sketch of the arguments here.

For any Dyck path $\gamma$ and any $i<j$  if
\beqlbl \label{overlap}
Exc(j)\subset Exc(i) \ \ \ \text{ then } \  \ \  \sigma_{\gamma}(j)<\sigma_{\gamma}(i)
\eeqlbl
 and if
\beqlbl \label{disjoint}
Exc(j)\cap Exc(i) =\emptyset \ \ \ \text{ then } \  \ \  \sigma_{\gamma}(i)<\sigma_{\gamma}(j).
\eeqlbl

If $\sigma\notin \sn(231),$ then there exists $i<j<k$ such that $\sigma(k)<\sigma(i)<\sigma(j)$.  Note that $\sigma(k)<\sigma(i)$ implies the $k$th up-step occurs before the end of the $i$th excursion.  Therefore the $j$th up-step also occurs before the end of the $i$th excursion which implies $\sigma(j)<\sigma(i)$ by \ref{overlap}, and $\sigma$ must be {\bf 231}-avoiding.

A dyck path, $\gamma,$ is uniquely defined by its peaks.  There is a one-to-one correspondence between peaks of $\gamma$ and left minimums of $\sigma$ $(\{(i,\sigma(i)), \sigma(i) < \sigma(j) \text{ for all } j > i\}).$  The left minimums of $\sigma$ are unique for 231-avoiding permutations, so $\sigma_\gamma$ is unique for each $\gamma$.   

\end{proof}

\section{Fixed Points for 231-avoiding permutations}

For a 231-avoiding permutation $\sigma \in S_n(231)$, let $\theta_{I}(\sigma)$ denote the number of fixed points of $\sigma$ contained in the subset $I\subset [n].$  Based on our bijection from Section \ref{dyckpathsection}, for a $\gamma\in \dyck{2n}$ and $\sigma = \sigma_\gamma$ and $\sigma(i)=i$ precisely when $l_i/2= h_i.$

\begin{theorem} \label{thm.231fplim}  

Fix $0< a<b<1$ and $\epsilon>0$. Let $\Gamma^n$ be chosen uniformly at random from $\dyck{2n}.$  Then

$$\lim_{n\to\infty}\prob\left(\left|\frac{1}{n^{1/4}}\theta_{[an,bn]}(\sigma_{\Gamma^n}) - \frac{1}{2\pi^{1/2}}\int_a^b \left(\frac{n^{1/2}}{\Gamma^n(2nt)}\right)^{3/2}dt \right| > \epsilon\right) = 0. $$

\end{theorem}

Using this theorem and a result from \cite{mp} we are able to prove our main result.

\begin{pfofthm}{\ref{geddewatanabe}}
Since $((2n)^{-1/2} \Gamma^n(2nt),0\leq t\leq 1) \rightarrow_d (\mathbbm{e}_t, 0\leq t\leq 1)$ in $C([0,1],\R)$, where $\mathbbm{e}_t$ denotes a standard Brownian excursion from 0 to 1, Theorem \ref{thm.231fplim} implies that for every fixed $0<a<b<1$, we have
\begin{equation} \label{eq dist for fatou} \frac{1}{n^{1/4}}\theta_{[an,bn]}(\sigma_{\Gamma^n}) \overset{d}{\longrightarrow} \frac{1}{2^{7/4}\pi^{1/2}}\int_a^b \mathbbm{e}_t^{-3/2}dt.\end{equation}
Our first step is to extend this convergence to $a=0$ and $b=1$.
For any $\delta \in [0,\frac12)$ we define the random variables
$$F_\delta = \frac{1}{2^{7/4}\pi^{1/2}}\int_\delta^{1-\delta} \mathbbm{e}_t^{-3/2} dt.$$

In \cite{chung1976excursions} the density function for the height of Brownian excursion at time $t \in [0,1]$  is determined to be
$$ \prob( \mathbbm{e}_t \in dh ) =  \frac{2^{1/2}|h|^2}{(\pi t^3 (1-t)^3)^{1/2}}\exp\bigg(-h^2/(2t(1-t))\bigg) dh$$
We can compute $\expect [F_0]$ by taking the expectation inside the integral and get
\begin{equation*}
\expect[F_0]  = \frac{1}{2^{7/4}\pi^{1/2}}\int_0^{1} \expect \left [ \mathbbm{e}_t^{-3/2} \right] dt
= \frac{ \text{ Gamma(1/4)}}{2 \pi^{1/2}}.
\end{equation*}
In \cite[Theorem 6.4]{mp} it is shown that
\begin{equation} \label{star star} \frac{\E \theta_{[1,n]}}{n^{1/4}} \to \frac{ \text{ Gamma(1/4)}}{2 \pi^{1/2}}. \end{equation} 
(In \cite{mp} the result is stated incorrectly and is off by a factor of four.)  Consequently, we have $n^{-1/4}\E  \theta_{[1,n]} \rightarrow \E F_0$.
By the monotone convergence theorem we have both $F_\delta \rightarrow_d F_0$ as $\delta\downarrow 0$ and $\E F_\delta \to \E F_0$ as $\delta \downarrow 0$.   Furthermore, by \eqref{eq dist for fatou} and the version of Fatou's Lemma for convergence in distribution (e.g. \cite[Lemma 4.11]{kallenberg2002foundations}) we have
$$\E F_\delta \leq \liminf_{n\to\infty} \frac{1}{n^{1/4}}\E \theta_{[\delta n, (1-\delta)n]}.$$
Thus, using the fact that $ \theta_{[\delta n, (1-\delta)n]} \leq \theta_{[1, n]}$, we have
\[\lim_{\delta \downarrow 0} \limsup_{n\to\infty} \E \left|  \frac{1}{n^{1/4}} \theta_{[1, n]}- \frac{1}{n^{1/4}} \theta_{[\delta n, (1-\delta)n]}\right|  \leq \lim_{\delta\downarrow 0} (\E F_0 - \E F_\delta) =0.\]
By \cite[Theorem 4.28]{kallenberg2002foundations} this implies that $n^{-1/4}\theta_{[1,n]}(\sigma_{\Gamma^n}) \rightarrow_d F_0$.  Since by \eqref{star star} this convergence happens in expectation as well, $(n^{-1/4}\theta_{[1,n]}(\sigma_{\Gamma^n}))_{n\geq 1}$ is uniformly integrable.  Since $ \theta_{[\delta n, (1-\delta)n]} \leq \theta_{[1, n]}$, for every fixed $\delta \in [0,1/2)$, the sequence $(n^{-1/4}\theta_{[\delta n,(1-\delta)n]}(\sigma_{\Gamma^n}))_{n\geq 1}$ is also uniformly integrable.  Consequently, we have $n^{-1/4}\E\theta_{[\delta n,(1-\delta)n]}(\sigma_{\Gamma^n}) \rightarrow \E F_\delta$.  We may now essentially repeat the argument for $a=0$ and $b=1$ to show that for every $t \in [0,1]$, we have 
\[ \frac{1}{n^{1/4}}\theta_{[1,tn]}(\sigma_{\Gamma^n}) \overset{d}{\longrightarrow} \frac{1}{2^{7/4}\pi^{1/2}}\int_0^t \mathbbm{e}_t^{-3/2}dt.\]
Since $t\mapsto \theta_{[1,tn]}(\sigma_{\Gamma^n})$ is non-decreasing, the convergence in probability of Theorem \ref{thm.231fplim} can be combined with standard arguments to give the desired joint process level convergence. 
\end{pfofthm}

Now we set up the notation necessary to prove Theorem \ref{thm.231fplim}.
We break the interval $[an,bn]$ up into subintervals of size about $n^{0.9}.$    In each of these intervals we will estimate the expected number of fixed points using the height of Dyck path at the start of the interval. Then we will bound the variance to show that with high probability the number of fixed points is close to the expected value. 

Label the intervals $I_k = [ a_k  , b_k)$ for $k \in [0,\cdots, K-1],$ where 
$K=\lfloor (b-a)n^{0.1}\rfloor$ and
$$a_k = \lfloor an + (k/K)(bn-an)\rfloor\text{  and }b_k = a_{k+1}.$$

Denote a sequence of heights $\alpha= \{\alpha^n_k\}_{k=0}^{K-1}$ and define 
$$\Omega^n(\alpha):= \bigcap_{k=0}^{K-1} \left \{\gamma \in \dyck{2n} | \gamma({v_{a_k}} ) = \alpha^n_k \right\}$$
where $v_{a_k}$ is the number of steps in the $\gamma$ up to and including the $a_{k}$th up-step.  
Note that $\Omega^n(\alpha)\cap \Omega^n(\alpha') = \emptyset$ if $\alpha \neq \alpha'.$  Let $\mathcal{A}$ denote the collection of all $\alpha.$  

\begin{definition}[A Proper Subset of $\dyck{2n}$]

We say a sequence of heights $\alpha = \{\alpha^n_{k}\}$ is proper if the following are satisfied for all $k=0,\dots,K$ 

\begin{itemize}
\item $ n^{0.499} <\alpha^n_k< n^{0.501}$ and 
\item $|\alpha^n_k - \alpha^n_{k+1}|<n^{0.451}$.
\end{itemize}
We say $\Omega^n(\alpha)$ is proper if $\alpha$ is proper.

\end{definition}

\begin{definition} \label{def.xihilirandom}
Recalling Definition \ref{def.xihili}, we define the random variables for a random path $\Gamma^n\in \dyck{2n}:$

\begin{itemize} 

\item $V^n _i :=$ number of steps up to and including the $i$th up-step.

\item $H^n_i :=\Gamma^n(V^n_i).$

\item $L^n_i := $ the length of the $i$th excursion.
\end{itemize}  

\end{definition}

Let $\mathcal{B}_n$ denote the collection of proper $\alpha\in\mathcal{A}$.  Most $\Gamma^n \in \dyck{2n}$ will be in some proper $\Omega^n(\alpha).$    

\begin{lemma} \label{notproper}

For $n$ sufficiently large, and $\Gamma^n$ be chosen uniformly at random from $\dyck{2n},$

$$\prob \left( \Gamma^n \in \bigcup_{\alpha\in\mathcal{B}_n} \Omega^n(\alpha)\right ) =1- o(1).$$

Moreover $$\prob\left(\bigcap_{i\in[an,bn]} \left\{n^{0.49}<H^n_i<n^{0.51}\right\} \Big | \Omega^n(\alpha) \right) > 1- e^{-n^{0.0001}}$$ for all proper $\Omega^n(\alpha).$

\end{lemma}

\begin{proof} 

The first statement follows from Lemmas \ref{isinwindow} and \ref{lowdeviation}.  The second statement follows by applying  Lemma \ref{allhigh} to the intervals $I_k$ for $0\leq k<K$.
\end{proof}

For a fixed sequence of heights $\alpha,$ let $\hat{k}(x) = \sup_{k} \{2a_k - \alpha^n_k \leq x \}.$  We define the following function $\rho_\alpha: [2an,2bn]\to [0,n]$
$$\rho_\alpha(x) := \alpha^n_{\hat{k}(x)}.$$  For $\gamma \in \Omega^n(\alpha),$ ${v_{a_k}} = 2a_k-\alpha^n_k$ and $\gamma( {v_{a_k}} ) = \rho_\alpha({v_{a_k}}) = \alpha^n_k.$  For most $\gamma \in \Omega^n(\alpha)$, 
$\gamma$ will be close to $\rho_\alpha.$

\begin{lemma} \label{gammarhoclose}

Fix $0<a<b<1,$ and $\epsilon> 0.$  For all $n$ sufficiently large,

$$\max_{\alpha \in \mathcal{B}_n}\left\{ \prob\left ( \sup_{t \in [a,b]}\left| \left(\frac{n^{1/2}}{\rho_\alpha(2nt)}\right)^{3/2} - \left(\frac{n^{1/2}}{\Gamma^n(2nt)}\right)^{3/2} \right  | > n^{-0.01} \Big| \Omega^n(\alpha) \right)\right\}< e^{-n^{0.001}}.$$

\end{lemma}

\begin{proof}

Let $\hat{k} = \hat{k}(2nt).$  By definition $v_{a_{\hat{k}}} < 2nt < v_{a_{\hat{k} + 1}}$ so 

$$|2nt - v_{a_{\hat{k}}} | < |2 a_{\hat{k}} - 2a_{\hat{k}+1} - \alpha^n_{\hat{k}} + \alpha^n_{\hat{k} +1}|< 2n^{0.9} + n^{0.451}< 3n^{0.9}.$$

Using Lemma \ref{allhigh} we obtain deviation bounds corresponding to all $i \in (a_k, a_{k+1})$ for $0\leq k < K.$  In particular we have for $t<3,$ $$\prob( |\Gamma^n( {v_{a_k}} + tn^{0.9}) - \Gamma^n({v_{a_k}})| > n^{0.46} |\Omega^n(\alpha)) <e^{-n^{0.0001}}.$$  By Lemma \ref{notproper}, we may also conclude that $\Gamma^n(2nt) > n^{0.49}-1$ with probability $1-e^{-0.0001},$ so with probability at least $1- 2e^{-n^{0.0001}}$

\begin{align*}
\left | \left(\frac{n^{1/2}}{\rho_\alpha(2nt)}\right)^{3/2} - \left(\frac{n^{1/2}}{\Gamma^n(2nt)}\right)^{3/2} \right | &< \left |\left(\frac{n^{1/2}}{\Gamma^n(2nt) ( \rho_\alpha(2nt)/ \Gamma^n(2nt))}\right)^{3/2} - \left(\frac{n^{1/2}}{\Gamma^n(2nt)}\right)^{3/2} \right |\\
 &< \left(\frac{n^{1/2}}{\Gamma^n(2nt)} \right)^{3/2} \left | 1- \frac{ 1 }{ 1 - n^{0.46}/\Gamma^n(2nt)} \right |^{3/2}\\
 & < n^{0.015} n^{-0.03} \\
&<  n^{-0.01}. 
\end{align*}

\end{proof}

\begin{lemma} \label{lemma.expectxab}
Fix $0<a<b<1.$ 
For all $n$ sufficiently large, 

$$
\max_{\alpha\in\mathcal{B}_n}  \left| n^{-1/4}\expect\left[\theta_{[an,bn]}\Big |\Omega^n(\alpha)\right] - 
\frac{1}{2\pi^{1/2}}\int_a^b \left(\frac{n^{1/2}}{\rho_\alpha(2nt)}\right)^{3/2}dt\right | < n^{-0.001}.
$$ 

\end{lemma}

\begin{lemma} \label{varcalc}
Fix $0<a<b<1.$ 
For all $n$ sufficiently large,

$$\max_{\alpha\in \mathcal{B}_n }\var\left[\theta_{[an,bn]} \big|\Omega^n(\alpha)\right] < n^{0.48}.$$

\end{lemma}

Because these bounds are uniform over all proper $\Omega^n(\alpha)$ we will drop the $\alpha$ where no confusion should arise.  We delay the proofs of these two lemmas until after the proof of Theorem \ref{thm.231fplim} as they are long and somewhat technical.  

\vspace{4pt}

\begin{proof}[Proof of Theorem \ref{thm.231fplim}] 
Fix a proper $\Omega^n$.  For all $n$ sufficiently large and $\Gamma^n$ chosen uniformly from $\Omega^n,$  by Lemma \ref{varcalc} and Chebyshev's inequality

$$\prob\left( \left |\theta_{[an,bn]}(\sigma_{\Gamma^n}) -\expect\left[\theta_{[an,bn]} |\Omega^n\right ] \right | > n^{0.005}n^{0.24}\Big| \Omega^n \right) < n^{-0.01}.$$

With Lemma \ref{lemma.expectxab} we have

$$\prob\left( \left |n^{-1/4}\theta_{[an,bn]} -\frac{1}{2\pi^{1/2}}\int_a^b \left(\frac{n^{1/2}}{\rho(2nt)}\right)^{3/2}dt\right| > n^{-0.005} + n^{-0.001} \Big| \Omega^n\right) < n^{-0.01}.$$ 
Combined with Lemma \ref{gammarhoclose}
\begin{equation}\label{propprob}
\prob\left( \left |n^{-1/4}\theta_{[an,bn]}(\sigma_\Gamma^n) -\frac{1}{2\pi^{1/2}}\int_a^b \left(\frac{n^{1/2}}{\Gamma^n(2nt)}\right)^{3/2}dt\right| > 2n^{-0.001} \Big| \Omega^n \right)  = \Delta(\Omega^n)\end{equation} where $\Delta(\Omega^n)= o(1)$ uniformly for all proper $\Omega^n.$  

Now consider $\Gamma^n$ chosen uniformly at random from $\dyck{2n}$.  

\begin{align*}
&\prob\left( \left |n^{-1/4}\theta_{[an,bn]}(\sigma_\Gamma^n) -\frac{1}{2\pi^{1/2}}\int_a^b \left(\frac{n^{1/2}}{\Gamma^n(2nt)}\right)^{3/2}dt\right| > 2n^{-0.001}  \right)\\
&\qquad \leq \prob\left( \Gamma^n \notin \cup_{\alpha\in\mathcal{B}} \Omega^n(\alpha)\right)  + \sum_{\alpha\in\mathcal{B}}\Delta(\Omega^n(\alpha))\prob( \Gamma^n \in \Omega^n(\alpha))\\
&\qquad=  o(1) + o(1)( 1- o(1))\\
&\qquad = o(1)
\end{align*}
by Lemma \ref{notproper}.
\end{proof}

\subsection{Proof of Lemma \ref{lemma.expectxab}}

For $i\in[an,bn]$ we have that $\theta_i:=\theta_{i}(\sigma_{\Gamma^n})$ is a 0-1 valued random variable where $$\prob(\theta_i = 1) = \prob( L^n_i/2 = H^n_i).$$  Let $I_k = \Iint{k} \cup \Iout{k}$ where $\Iout{k}$ consists of the $2n^{0.6}$ values both directly after $a_k$ or directly before $a_{k+1}$ and $\Iint{k}$ is the rest of $I_k.$  

\begin{lemma} \label{expiintk}

Fix $0<a<b<1.$  For all proper $\Omega^n$ and for each $k,$ and $i\in\Iint{k}.$

$$\expect[\theta_i|\Omega^n] = \frac{1}{2\pi^{1/2} (\alpha^n_k)^{3/2}}(1+\Delta),$$ where $\Delta = \Delta(i,k,\Omega^n) = o(n^{-0.01})$ uniformly in $i, k$ and proper $\Omega^n.$

\end{lemma}

\begin{proof} 
For each $k$ the and each $i$ in $\Iint{k}$ the conditions for Lemma \ref{problihiapprox} are satisfied since $\Omega^n$ is proper.   Therefore

$$\expect[\theta_i |\Omega^n] = \frac{1}{2\pi^{1/2} (\alpha^n_k)^{3/2}}(1+\Delta)$$ as desired.

\end{proof}

For $\Iout{k}$ we look at $\expect[\theta_{\Iout{k}}|\Omega^n]$ as a whole rather than computing $\expect[\theta_i|\Omega^n]$ for each individual $i.$

\begin{lemma} \label{fixioutk}

Fix $0<a<b<1$ and proper $\Omega^n.$  For all $\gamma\in\Omega^n$ $$\sum_k \expect\left [\theta_{\Iout{k}} | \Omega^n\right] \leq 6n^{0.21}.$$

\end{lemma}

\begin{proof} 

For each $k$, $\Iout{k}$ consists of two intervals of length $2n^{0.6}$ which can be covered by less than $5n^{0.6}/n^{0.49}$ subintervals of length $n^{0.49}$.  As $\Omega^n$ is proper, Lemma \ref{notproper} says $h_i>n^{0.49}$ for $i\in [an,bn]$ with probability $e^{-n^{0.0001}}.$  Then by Lemma \ref{fixsmall} each of the subintervals has at most one fixed point.  Then $\expect\left [\theta_{\Iout{k}}|\Omega^n\right ] \leq 5n^{0.11} +2n^{0.6} e^{-n^{0.00001}}$ for each $0\leq k< K < n^{0.1}.$  Adding them up proves the lemma. 
\end{proof}

\begin{lemma}  \label{expfixab}

For fixed $0<a<b<1$ and proper $\Omega^n,$ 

$$
\expect[ \theta_{[an,bn]} | \Omega^n ] 
=  ( 1+\Delta )\sum_{j=\lfloor an \rfloor }^{\lfloor bn \rfloor} \frac{1}{2 \pi^{1/2} \rho(V^n_j)^{3/2 } },
$$ where $\Delta = o(n^{-0.01})$ uniformly in proper $\Omega^n$
\end{lemma}

\begin{proof}  

By linearity of expectation:

\begin{align*}
\expect[ \theta_{[an,bn]} | \Omega^n ]& = \sum_{k= 0}^{K-1} \sum_{i' \in I_k}\expect[ \theta_{i'}|\Omega^n]=\sum_{k=0}^{K-1} \sum_{i=0}^{|I_k|-1} \expect[\theta_{a_k + i}|\Omega^n]
 \end{align*}

For each $k,$ and $a_k + i \in \Iint{k},$ we can apply Lemma \ref{expiintk} to conclude 
$$\expect \left[\theta_{a_k + i}  \big | \Omega^n\right] = \frac{1}{2\pi^{1/2} (\alpha^n_k)^{3/2}}(1+ \Delta(i,k,\Omega^n))$$   
where $ \Delta(i,k,\Omega^n) = o(n^{-0.01})$ uniformly in $i,k$ and proper $\Omega^n.$

By Lemma \ref{notproper} we know the paths are high enough to apply Lemma \ref{fixioutk} to show that $\expect[\sum_k \theta_{\Iout{k} }] <n^{0.22}.$  On the other hand $\alpha^n_k < n^{0.51}$ implies 

\begin{align*}\sum_k\expect[\theta_{\Iint{k} }|\Omega^n] &= \sum_k \sum_{i \in \Iint{k} } \frac{1}{2{\pi^{1/2} (\alpha^n_k)^{3/2}}}(1+\Delta(i,k,\Omega^n))\\
& > \sum_k \sum_{i \in \Iint{k} } n^{-0.765}\\
& > n^{0.23}\\
\end{align*} so the contribution from $\sum_k \expect[\theta_{\Iout{k}}|\Omega^n]$ is dominated by $\sum_k\expect[\theta_{\Iint{k}}|\Omega^n].$  Then by Lemmas \ref{expiintk} and \ref{fixioutk}

\begin{align*}
\expect[\theta_{[an,bn]}|\Omega^n] &=\expect[\theta_{\Iout{k}} |\Omega^n] +  \sum_k \expect[ \theta_{\Iint{k}}|\Omega^n] \\
&=5n^{0.21} + \sum_k \frac{|\Iint{k}|}{2\pi^{1/2}(\alpha^n_k)^{3/2}} (1+\Delta(k,\Omega^n))\\
&= \left(1+ \Delta(\Omega^n)\right)\sum_k \frac{|\Iint{k}|}{2\pi^{1/2}(\alpha^n_k)^{3/2}} 
\end{align*}
where $\Delta(k,\Omega^n)= o(n^{-0.01})$ uniformly in $k$ and proper $\Omega^n$ and $\Delta(\Omega^n) = o(n^{0.01})$ uniformly in $\Omega^n.$  

For each $k$, $|I_k| = (1+O(n^{-0.3})) |\Iint{k}|$  by the definitions of $I_k$ and $\Iint{k}$.  Then the above expression becomes 
\begin{align*}\expect[\theta_{[an,bn]}|\Omega^n]=&(1+\Delta'(\Omega^n)) \sum_k\frac{ |I_k|}{2\pi^{1/2}( \alpha^n_{k})^{3/2}}\\
 =& (1+\Delta'(\Omega^n) ) \sum_k \sum_{j\in I_k}\frac{1}{2\pi^{1/2} (\alpha^n_{k})^{3/2}}.
 \end{align*}  with $\Delta'(\Omega^n) = o(n^{-0.01})$ uniformly in all proper $\Omega^n$.  For $j \in I_k,$ $\rho(V^n_j)  = \alpha^n_k,$  finishing the proof. 
\end{proof}

\begin{proof}[Proof of Lemma \ref{lemma.expectxab}]

By Lemma \ref{expfixab} we can write the conditional expectation of $\theta_{[an,bn]}$ as

$$\expect[ \theta_{[an,bn]} | \Omega^n ] = ( 1+\Delta )\sum_{i= \lfloor an \rfloor}^{\lfloor  bn \rfloor} \frac{1}{2 \pi^{1/2}\rho(V^n_i)^{3/2} } $$ where $\Delta=o(n^{-0.01})$ uniformly in all proper $\Omega^n.$  Converting the sum into an integral we have
$$\expect[ \theta_{[an,bn]} | \Omega^n ] = (1+\Delta)\int_{an}^{bn} \frac{1}{2 \pi^{1/2} \rho(V^n_{\lfloor u\rfloor})^{3/2}} du.$$
The change of variables $nt = u$ gives
$$\expect[\theta_{[an,bn]} |\Omega^n] = (1+\Delta) \int_a^b \frac{ 1}{2\pi^{1/2} \rho(V^n_{\lfloor nt\rfloor})^{3/2}} ndt.$$

Since $\Omega^n$ is proper, $|2nt - V^n_{\lfloor nt \rfloor} |< n^{0.51}.$  Therefore either $\hat{k}(V^n_{\lfloor nt \rfloor }) = \hat{k}(2nt)$ or $\hat{k}(2nt)-1.$  In either case by properness of $\Omega^n,$ $|\rho(V^n_{\lfloor nt\rfloor }) - \rho(2nt)| < n^{0.451}$ and $\rho(2nt) > n^{0.499}$ for $t\in[a,b]$ so 
$$\frac{1}{\rho(V^n_{nt})^{3/2}} = (1 + o(n^{-0.001}))\frac{1}{\rho(2nt)^{3/2}}.$$
Scaling by $n^{1/4}$ completes the proof.  
\end{proof}

\subsection{Proof of Lemma \ref{varcalc}}

Now that we have the conditional expectation $\expect[\theta^{int}_{[an,bn]}|\Omega^n],$ we will bound the conditional variance, $\var[\theta_{[an,bn]}|\Omega^n].$ 

Our basic variance equation is

$$\var[\theta_{[an,bn]}|\Omega^n] =\sum_{i,j}  \expect[ \theta_i\theta_j|\Omega^n] - \expect [\theta_i|\Omega^n] \expect[\theta_j|\Omega^n].$$

The key to bounding the conditional variance for a proper $\Omega^n$ is understanding $\expect[\theta_i\theta_j|\Omega^n]$ for various ranges of $i$ and $j$.  We cover $[an,bn]^2$ with $\cup_{l=1}^5 B_l$ where each $B_l$ is defined as follows:  

\begin{itemize}

\item $B_1 = \cup_k \cup_{k'}  \Iout{k}\times  \Iout{k'},$ 
\item $B_2 = \cup_k \cup_{k'} \{\Iint{k}\times \Iout{k'}\}\bigcup \{\Iint{k'}\times \Iout{k}\},$ 
\item $B_3 = \cup_k \cup _{k'\neq k} \Iint{k}\times \Iint{k'},$
\item $B_4 = \cup _k \Iint{k}\times \{ j\in \Iint{k} s.t. |j-i| \leq 2n^{0.6}\},$
\item $B_5 = \cup _k \Iint{k}\times \{ j\in \Iint{k} s.t. |j-i| > 2n^{0.6}\}.$

\end{itemize}

Consider the property $$P  = \bigcap_{i\in [an,bn]} \{ n^{0.49} < H^n_i < n^{0.51} \}.$$  
For each $B_l$ we will show that $$\sum_{(i,j) \in B_l} \expect[\theta_i\theta_j\indc_P |\Omega^n] - \expect[\theta_i\indc_P|\Omega^n]\expect[\theta_j\indc_P|\Omega^n] = o(n^{0.48}).$$  Hence the total variance, $\var[ \theta_{[an,bn]}\indc_P|\Omega^n]$ is $o(n^{0.48}).$

The following lemma allows us extend this bound to $\var[\theta_{[an,bn]}|\Omega^n].$

\begin{lemma} \label{notP} 

Fix $0<a<b<1.$  For $n$ sufficiently large and proper $\Omega^n$

$$\sum_{(i,j) \in [an,bn]^2} \expect[\theta_i\theta_j|\Omega^n] =  (1+\Delta) \sum_{(i,j)\in[an,bn]^2}\expect[ \theta_i\theta_j\indc_P |\Omega^n ] $$
with $\Delta = o(1)$ uniformly in for all proper $\Omega^n$.

\end{lemma}

\begin{proof} 

Since $\theta_i\theta_j \indc_{P^C} < \indc_{P^C},$ by Lemma \ref{notproper} we have $$\expect[\theta_i\theta_j\indc_{P^C}|\Omega^n] \leq \expect[\indc_{P^C}|\Omega^n] \leq n^3\exp^{-n^{0.0001}}.$$  

Noting that $\theta_i\theta_j = \theta_i\theta_j\indc_P + \theta_i\theta_j\indc_{P^C}$ and taking expectation gives

$$\expect[\theta_i\theta_j| \Omega^n]  \leq \expect [\theta_i\theta_j \indc_P |\Omega^n] + \expect[\indc_{P^C}|\Omega^n] \leq \expect[\theta_i\theta_j\indc_P|\Omega^n] + n^3\exp(-n^{0.0001}).$$ 

Summing over $i$ and $j$ completes the proof.
\end{proof}

For the following series of lemmas we will assume the following standard hypotheses.

\begin{itemize}

\item Fix $0<a<b<1.$
\item $\Omega^n \subset \dyck{2n}$ is proper.  
\item Let $n$ be large enough such that $n^5e^{-n^{0.0001}} < o(1).$

\end{itemize}

\begin{lemma} \label{B1}

Assuming the standard hypotheses, 

$$\sum_{(i,j)\in B_1 }\expect[\theta_i\theta_j\indc_P|\Omega^n] < n^{0.47}.$$

\end{lemma}

\begin{proof} 
For fixed $k'$ we may use Lemma \ref{fixsmall} to show $\sum_{j\in\Iout{k'}}\expect[\theta_j| *]< 2n^{0.6-0.49} $ no matter the conditions given by $*$.  In particular we have

\begin{align*}
\sum_{j\in \Iout{k'}} \expect[ \theta_i\theta_j \indc_P| \Omega^n ] &\leq \sum_{j\in \Iout{k'}} \expect[\theta_j | \Omega^n , \theta_i\indc_P=1]\expect[\theta_i\indc_P|\Omega^n]\\
&\leq  \expect[\theta_i\indc_P|\Omega^n] 4 n^{0.6-0.49}.
\end{align*}
Then
\begin{align*}
\sum_k\sum_{k'} \sum_{i \in \Iout{k} } \sum_{j \in \Iout{k'} }\expect[\theta_i\theta_j\indc_P|\Omega^n]& \leq \sum_k\sum_{k'}\sum_{i \in \Iout{k}} \expect[\theta_i\indc_P|\Omega^n] 2n^{0.6-0.49} \\
&\leq \sum_k \sum_{k'} 4n^{0.22}\\
&\leq n^{0.42}. \\
\end{align*}

\end{proof}

\begin{lemma} \label{B2}

Assuming the standard hypotheses,

$$\sum_{(i,j) \in B_2 }\expect[\theta_i\theta_j\indc_P|\Omega^n] < 6n^{0.47}.$$

\end{lemma}

\begin{proof} 

This follows the proof of Lemma \ref{B1} closely.   By Lemma \ref{fixsmall} 

$$
\sum_k\sum_{k'} \sum_{i \in \Iint{k} } \sum_{j \in \Iout{k'} }\expect\left[\theta_i\theta_j\indc_P|\Omega^n\right] \leq  \sum_k\sum_{k'}\sum_{i \in \Iint{k}} \expect[\theta_i\indc_P|\Omega^n] 2n^{0.6-0.49} .
$$

For each $k$ and each $i \in \Iint{k}$,  Lemma \ref{problihiapprox} and the properness of $\Omega^n$ imply that  $\expect[\theta_i\indc_P |\Omega^n] < (n^{-0.495})^{3/2} < n^{-0.74}.$  Then

$$\sum_k\sum_{k'}\sum_{i\in \Iint{k}}3n^{0.11}\expect[\theta_i\indc_P|\Omega^n] \leq 3n^{0.1}n^{0.1}n^{0.9}n^{0.11}n^{-0.74} < 3n^{0.47}.$$

Changing the roles of $i$ and $j$ and doubling the upper bounded completes the proof.
\end{proof}

\begin{lemma} \label{B3}

Assuming the standard hypotheses,

$$\sum_{i,j\in B_3} \expect\left[\theta_i\theta_j\indc_P| \Omega^n\right ]  - \expect[\theta_i\indc_P|\Omega^n]\expect[\theta_j\indc_P|\Omega^n]=0.$$

\end{lemma}

\begin{proof}

The flavor of this proof is somewhat different from the previous lemmas.  Without loss of generality we may assume that $k<k'.$  

If $\theta_i\indc_P=1,$ then the corresponding $i$th excursion will end before the $a_{k'}$th excursion begins as $$L^n_i/2 = H^n_i < n^{0.51} < 2n^{0.6}$$ and $a_{k'}> i + |\Iout{k}|.$  Therefore, for $i \in I_k$ and $j \in I_{k'},$ $$\expect [\theta_i\theta_j \indc_P| \Omega^n]  = \expect [\theta_i\indc_P|\Omega^n]\expect[\theta_j\indc_P|\Omega^n]$$ and
$$
\sum_{(i,j)\in B_3}\left( \expect[\theta_i\theta_j\indc_P| \Omega^n ]  -\expect[\theta_i\indc_P|\Omega^n]\expect[\theta_j\indc_P|\Omega^n]\right) = 0.
$$

\end{proof}

\begin{lemma}  \label{B4}

Assuming the standard hypotheses,

$$\sum_{(i,j) \in B_4}  \expect[\theta_i\theta_j\indc_P |\Omega^n ] < n^{0.47}.$$

\end{lemma}

\begin{proof}

By Lemma \ref{fixsmall} 

$$\sum_{|i-j| < 2n^{0.6} } \expect[\theta_i\theta_j\indc_P | \Omega^n ] \leq 5n^{0.6-0.49}\expect[\theta_i\indc_P|\Omega^n].$$
For each $k,$  $|\Iint{k}| \leq n^{0.9}$ so 
$$\sum_k \sum_{i\in\Iint{k}}5n^{0.11}n^{-0.73} \leq 5n^{0.1+0.9 + 0.11-0.735} < n^{0.47}.$$
\end{proof}

The last possibility is the one which requires the most care.  

\begin{lemma}  \label{B5}

Assuming the standard hypotheses,  

$$\sum_{(i,j)\in B_5} \expect[\theta_i\theta_j\indc_P|\Omega^n] <n^{0.47}.$$

\end{lemma}

\begin{proof}

We proceed in a manner similar to Lemmas \ref{problihi} and \ref{problihicondapprox}.  For $(i,j) \in B_5,$ with $i \in \Iint{k}.$

$$\expect[\theta_i\theta_j\indc_P | \Omega^n ] = \sum_{n^{0.49}<h<n^{0.51}}\sum_{n^{0.49}<h'<n^{0.51}} G(i,j,h,h',a_k,a_{k+1},\alpha^n_k,\alpha^n_{k+1}).$$where 

$$G(i,j,h,h',a_k, a_{k+1},\alpha^n_k,\alpha^n_{k+1}) $$
$$= C_{h-1} C_{h'-1} \left|\cE_{{v_{a_k}},\alpha^n_k}^{2i-h,h}\right|\left|\cE_{2i-h+2h-1,h-1}^{2j-h',h'}\right| \left|\cE_{2j-h'+2h'-1,h'-1}^{v_{a_{k+1}},\alpha^n_{k+1}}\right| \left|\cE_{{v_{a_k}},\alpha^n_k}^{v_{a_{k+1}},\alpha^n_{k+1}}\right|^{-1}.$$

For fixed $i,j$ and $k$ there are two cases to consider for values $h$ and $h'$.  One where we can use Lemma \ref{pathcount} for each each section of the path, and one where we bound $G$ by $e^{-n^{0.001}}$ using Lemma \ref{theelemma}.

Define the set of pairs of heights $D_{i,j,k}$ such that for $(h,h')\in D_{i,j,k},$ Lemma \ref{pathcount} is valid for each of the path sections.  For $(h,h') \notin D_{i,j,k} $ the contribution to $\expect[\theta_i\theta_j\indc_P|\Omega^n ] $ is bounded by $e^{-n^{0.001}}.$  Otherwise

$$\expect[\theta_i\theta_j\indc_P|\Omega^n] \leq \sum_{(h,h')\in D_{i,j,k}}  \sqrt{\frac{ n^{0.9}}{ i(j-i)(n^{0.9}-j)h^3{h'}^3 }} F(i,j,h,h',\alpha^n_k,\alpha^n_{k+1})$$ where 

$$F(i,j,h,h',\alpha^n_k,\alpha^n_{k+1}) = \exp \left( -\frac{(h-\alpha^n_k)^2}{4i} -\frac{(h'-h)^2}{4(j-i)}-\frac{(\alpha^n_{k+1} - h')^2}{4(n^{0.9}-j)} + \frac{(\alpha^n_{k+1}-\alpha^n_k)^2}{4n^{0.9}} \right).$$

For $h$ and $h'\in (n^{0.49},n^{0.51})$  we may replace $\frac{1}{(hh')^{3/2}}$ with $n^{-1.47}.$  By Lemma \ref{uglyvarcalc} there exist a large constant $C$ such that 

$$\sum_{h,h'} F(i,j,h,h',\alpha^n_k,\alpha^n_{k+1}) < C\sqrt{\frac{i(j-i)(n^{0.9}-j)}{n^{0.9}}}$$ uniformly over all choices of $(i,j)\in B_5$ and $\alpha^n_k$ and $\alpha^n_{k+1}$ from a proper sequence of heights.

This gives $$\sum_{(i,j)\in B_5} \expect[\theta_i\theta_j \indc_P |\Omega^n ] \leq \sum_{(i,j) \in B_5} Cn^{-1.47} \leq Cn^{1.9}n^{-1.47} <n^{0.47}.$$
\end{proof}

\begin{proof}[Proof of Lemma \ref{varcalc}]
Lemmas \ref{B1}, \ref{B2}, \ref{B3}, \ref{B4}, and \ref{B5} combine to  show that the variance $\var[\theta_{[an,bn]}\indc_P|\Omega^n]<8n^{0.47}$.  Lemma \ref{notP} finishes the proof.
\end{proof}

\begin{pfofcor}{\ref{sixteencandles}}
We follow the proof of Theorem \ref{geddewatanabe} with very minor changes. In particular we use Corollaries \ref{problihiapprox2} and \ref{fixedsmall} in place of Lemmas \ref{problihiapprox} and \ref{fixsmall}. Everything else follows in exactly the same manner when $\alpha<.49$. For $\alpha \in [.49,.5)$ we follow the proof of Theorem \ref{geddewatanabe} changing the exponents to $.5 \pm \delta$. We leave the details to the reader.
\end{pfofcor}

The following appendix contains various technical lemmas that will be used throughout the paper.  The statements of the lemmas are similar to results found elsewhere, but modified for use in this paper.  

\section{Appendix A: Technical Lemmas}
\label {appendixb}

We begin with a useful Lemma that will help count non-negative lattice paths between points.  Let $\mathcal{A}_n$ denote the set of points $(i,m) \in \Z^2$ such that $0<n^{0.6}<i<n$ and $|m| < i^{0.6}.$

\begin{lemma}  \label{theelemma}

For $(i,m) \in \mathcal{A}_n.$

\beq{ 2i-m \choose i  } = \frac{(1+ \Delta(i,m))4^{i}}{2^m\sqrt{\pi i}}e^{-\frac{m^2}{4i }}
\eeq where $\Delta(i,m) = o(n^{-0.1})$ uniformly in $i$ and $m$ in $\mathcal{A}_n$.  

For $i> n^{0.6}$ and $|m| > i^{0.6}$

$${2i -m \choose i} \leq \frac{4^i}{2^m}e^{-n^{0.1}}.$$
\end{lemma}

\begin{proof} 
This first equality follows from $\triangleright$ IX.1 on page 615 of Flajolet and Sedgewick \cite{FlSe09}.  For the second equality we let $m = i^{0.6} + r$ or $m = -i^{0.6} -r$ for some $r>0.$

$$ {2i - m \choose i }  = {2i - i^{0.6} \choose i } \prod_{k = 0}^{r-1} \frac{ i - i^{0.6} -k }{2i - i^{0.6} - k}$$
$$ \leq \frac{4^i}{2^{i^{0.6} +r}}e^{-i^{0.2}/4} \prod_{k=0}^{r-1}\frac{ 2i - 2i^{0.6} - 2k }{2i - i^{0.6} - k}.$$
$$\leq \frac{4^i}{2^m}e^{-n^{0.12}/4}.$$
A similar computation holds for $m = -i^{0.6} -r.$
\end{proof}

Consider a lattice path starting at $(v_0,h_0).$   Recall Definition \ref{def.xihili}.  We may extend those definitions to general lattice paths with a slight modification.  The definitions $v_i$ and $h_i$ remain the same, the position and the height after the $i$th up-step from the start of the path.  For $l_i$ we do not necessarily have an excursion.  If the path never returns below $h_i$ at some time later than $v_i$ then we say that $l_i = \infty.$  

\begin{lemma} \label{pathcount}

Suppose $(i,m)\in \mathcal{A}_n$ and $h = h_0 + m.$  The number of lattice paths from $(v_0,h_0)$ to $(v_i,h)$ is given by

$${ 2(i-1) - (m-1) \choose i-1} = \frac{4^{i}}{2^{m+1}\sqrt{\pi i}}\exp\left( -\frac{m^2}{4i} \right) (1+\Delta(i,m))$$
where $\Delta(i,m)$ as defined in Lemma \ref{theelemma}.
\end{lemma}

\begin{proof} 

Let $i$ and $d$ denote the number of up and down steps respectively in a lattice path up to and including the $i$th up-step.  We denote the total number of steps by $v_i=i+d$.  The change in height for that path is given by $h-h_0 = m= i-d$.  Then $v_i=2i-m$ counts the total number of steps.  The second to last position of the path is $(x + 2i -m -1,h_0+m-1)$ since the $v_i$th step is assumed to be an up-step.  Therefore the total number of lattices paths from $(v_0,h_0)$ to $(v_0+2i-m-1,h_0 +m -1)$  is counted by Lemma \ref{theelemma}, giving the equation found in Lemma \ref{pathcount}

\end{proof}

Now that we can accurately count the number of lattice paths from one point to another we can count the number of non-negative paths between two points.  For a pair of points $(v_0,h_0)$ and $(v_i,h_i)$ we let $\mathcal{E}_{v_0,h_0}^{v_i,h_i}$ denote the set of non-negative lattice paths ending with an up-step between the two points.

\begin{lemma} \label{pathcountapprox}

For $(i,m) \in \mathcal{A}_n$, let $h_0> n^{0.499},$ $v_i = v_0 + 2i-m,$ and $h_i = h_0+m.$  Then

$$\left|\cE_{v_0,h_0}^{v_i,h_i}\right| =  \frac{4^{i}}{2^{m+1}\sqrt{\pi i}}\exp\left( -\frac{m^2}{4i} \right) (1+\Delta'(i,m))$$ where $\Delta'(i,m)=o(n^{-0.1})$ uniformly in $i,m \in A_n.$

\end{lemma}

\begin{proof} 

We count using standard ballot counting arguments. 

$$\left|\cE_{v_0,h_0}^{v_0+2i-m,h_0+m}\right| = {2i-m-1 \choose i-1} - {2i -m -1 \choose  i + h_0  }.$$

For $h_0 >n^{0.49},$ $${2i-m-1 \choose i+h_0 } < {2i-m-1 \choose i-1} \exp(-(h_0)^2/2i).$$  
Moreover $(h_0)^2/2i >n^{0.07},$  so    
$$\left|\cE_{v_0,h_0}^{v_i,h_i}\right| = {2i-m-1\choose i-1} ( 1+ O( \exp(-n^{0.06}) ) ) =  \frac{4^{i}}{2^{m+1}\sqrt{\pi i}}\exp\left( -\frac{m^2}{4i} \right) (1+\Delta'(i,m))$$ as desired.
\end{proof}

For paths chosen uniformly from $\cE_{v_0,h_0}^{v_j,h_j}$ for $0<i<j$ we would like to know for various values of $i$ and $h$ how many of these path go through the point $(v_i,h)$ after the $i$th up-step.  Given $\Gamma^n \in \cE_{v_0,h_0}^{v_j,h_j}$ chosen uniformly at random what is the probability that $H^n_i = h?$

\begin{lemma} \label{probxih}

Fix $v_0, h_0, h_j$ and $h_j$.   Let $X^n$ be chosen uniformly from $\cE_{v_0,h_0}^{v_j,h_j}$. For $h>0 $ and $0<i<j$, 

$$\prob( H^n_i = h) =  \left| \cE_{v_0,h_0}^{2i - (h-h_0),h} \right | \left| \cE_{2i-(h-h_0),h}^{v_j,h_j}\right| \left| \cE_{v_0,h_0}^{v_j,h_j}\right|^{-1}.$$

\end{lemma}

\begin{proof} 
Any path in $\cE_{v_0,h_0}^{v_j,h_j}$ can be decomposed uniquely into a concatenation of two paths, one in $\cE_{v_0,h_0}^{v_i,h_i}$  and the other in $\cE_{v_i,h_i}^{v_j,h_j}$ for some appropriate values of $v_i$ and $h_i$ that satisfy $v_i = v_0 + 2i - (h_i-h_0).$  If $h_i = h,$ then $v_i = v_0 + 2i-(h-h_0).$  The set $\left\{X^n \in \cE_{v_0,h_0}^{v_j,h_j} | h_i=h\right\}$ is in bijection with $\cE_{v_0,h_0}^{v_i,h} \times \cE_{v_i,h}^{v_j,h_j}.$  Then 

\begin{align*}
\prob( H^n_i = h) =& \left|\left\{\Gamma^n \in \cE_{v_0,h_0}^{v_j,h_j} | H^n_i=h\right\}\right|\left |\cE_{v_0,h_0}^{v_j,h_j}\right|^{-1}\\
 =& \left| \cE_{v_0,h_0}^{2i-(h-h_0),h} \right | \left| \cE_{2i-(h-h_0),h}^{v_j,h_j}\right| \left| \cE_{v_0,h_0}^{v_j,h_j}\right|^{-1}
 \end{align*}
as desired.
\end{proof}

\begin{lemma} \label{problihicond}

Fix $v_0, h_0, v_j,$ with $h_j$ such that $v_j - v_0 =2j - (h_j-h_0)$ as above.  Also fix $i,h> 0$ and $i < j-h.$   For $X^n$ chosen uniformly from $\cE_{v_0,h_0}^{v_j,h_j},$

$$\prob(L^n_i/2 = h | H^n_i = h  )= C_{h-1} \left| \cE_{2i-(h-h_0)+2h-1,h-1}^{v_j,h_j}\right| \left| \cE_{2i-(h-h_0),h}^{v_j,h_j}\right|^{-1}.$$

\end{lemma} 

\begin{proof} 

Let $X^n \in \cE_{v_0,h_0}^{v_j,h_j}$ also satisfy $H^n_i = h$.  From Lemma \ref{probxih} there are precisely 
$$|\cE_{v_0,h_0}^{2i-(h-h_0),h}||\cE_{2i-(h-h_0),h}^{v_j,h_j}|$$ such paths.  Each of these paths that satisfies $L^n_i/2 = h$ has a unique decomposition into three parts:

\begin{itemize}
\item A path $X^n_1 \in \cE_{v_0,h_0}^{2i-(h-h_0),h}$,
\item an excursion $X^n_2$ from $(2i-(h-h_0)_i,h)$ to $(2i-(h-h_0) + 2h-2,h)$, staying above $h-1,$
\item a single down-step from $(2i-(h-h_0)+2h-2,h)$ to $(2i-(h-h_0)+2h-1,h-1)$,
\item and a path $X^n_3 \in \cE_{2i-(h-h_0) + 2h-1,h-1}^{v_j,h_j}.$
\end{itemize}

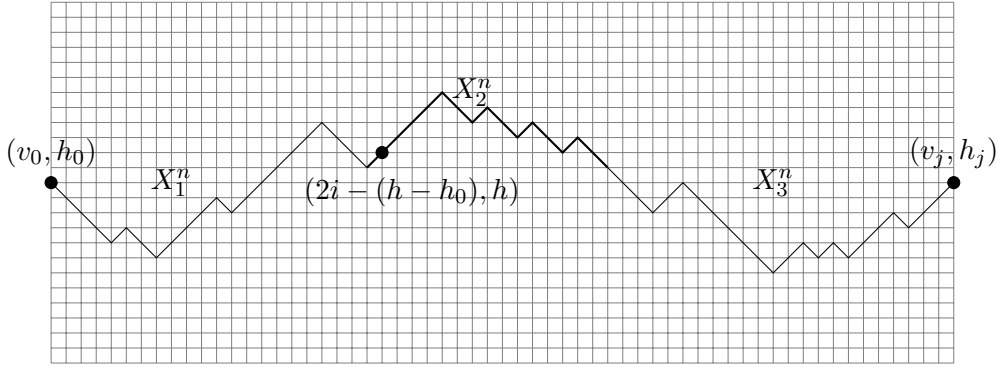
\begin{figure}
\centering
\begin{tikzpicture}[scale=.8]

\draw [step=.25, help lines,  gray] (0,0) grid (15,6) ;

\node at (0,3.5) {$(v_0,h_0)$};

\draw [fill] (0,3) circle (.1cm);

\draw (0, 3)--(1/4, 11/4)--(1/2, 5/2)--(3/4, 9/4)--(1, 2)--(5/4, 9/4)--(3/2, 2)--(7/4, 7/4)--(2, 2)--(9/4, 9/4)--(5/2, 5/2)--(11/4, 11/4)--(3, 5/2)--(13/4, 11/4)--(7/2, 3)--(15/4, 13/4)--(4, 7/2)--(17/4, 15/4)--(9/2, 4)--(19/4, 15/4)--(5, 7/2)--(21/4, 13/4);

\draw [thick] (21/4, 13/4)--(11/2,7/2);

\draw [fill] (11/2,7/2) circle (.1cm);

\node at (6,2.85) {$(2i-(h-h_0),h)$};

\draw [thick] (11/2, 7/2)--(23/4, 15/4)--(6, 4)--(25/4, 17/4)--(13/2, 9/2)--(27/4, 17/4)--(7, 4)--(29/4, 17/4)--(15/2, 4)--(31/4, 15/4)--(8, 4)--(33/4, 15/4)--(17/2, 7/2)--(35/4, 15/4)--(9, 7/2);

\draw [thick] (9,7/2) -- (37/4,13/4);

\draw (37/4, 13/4)--(19/2, 3)--(39/4, 11/4)--(10, 5/2)--(41/4, 11/4)--(21/2, 3)--(43/4, 11/4)--(11, 5/2)--(45/4, 9/4)--(23/2, 2)--(47/4, 7/4)--(12, 3/2)--(49/4, 7/4)--(25/2, 2)--(51/4, 7/4)--(13, 2)--(53/4, 7/4)--(27/2, 2)--(55/4, 9/4)--(14, 5/2)--(57/4, 9/4)--(29/2, 5/2)--(59/4, 11/4)--(15,3);

\draw [fill] (15, 3) circle (.1cm);

\node at (15,3.5) {$(v_j,h_j)$};

\node at (2,3) {$X^n_1$};
\node at (7,4.5) {$X^n_2$};
\node at (12,3) {$X^n_3$};

\end{tikzpicture}

\caption{ Decomposition of $X^n$ into $X^n_1$, $X^n_2$, and $X^n_3.$}

\end{figure}

The choice of $X^n_1, X^n_2, $ and $X^n_3$ uniquely determines $X^n.$  There are $$\left|\cE_{v_0,h_0}^{2i-(h-h_0),h} \right|, C_{h-1}, \text{ and }\left|  \cE_{2i-(h-h_0)+2h-1,h-1}^{v_j,h_j}\right|$$ such choices for $X^n_1, X^n_2,$ and $X^n_3$ respectively.  Therefore

\begin{align*}
\prob(L^n_i/2 = h | H^n_i = h) &=\left|\cE_{v_0,h_0}^{2i-(h-h_0),h} \right|C_{h-1} \left| \cE_{2i-(h-h_0)+2h-1,h-1}^{v_j,h_j}\right| \left(\left| \cE_{v_0,h_0}^{2i-(h-h_0),h} \right | \left| \cE_{2i-(h-h_0),h}^{v_j,h_j}\right|\right)^{-1}\\
&= C_{h-1} \left| \cE_{2i-(h-h_0)+2h-1,h-1}^{v_j,h_j}\right|  \left| \cE_{2i-(h-h_0),h}^{v_j,h_j}\right|^{-1}
\end{align*}

\end{proof}

\begin{lemma} \label{problihi}
For $X^n \in \cE_{v_0,h_0}^{v_j,h_j}$ chosen uniformly at random and $0<i<j,$

$$\prob(L^n_i/2 = H^n_i ) = \sum_{h=\max(0,h_j-(j-i))}^{h_0 + i} C_{h-1} \left| \cE_{v_0,h_0}^{2i-(h-h_0),h} \right | \left| \cE_{2i-(h-h_0)+2h-1,h-1}^{v_j,h_j}\right| \left| \cE_{v_0,h_0}^{v_j,h_j}\right|^{-1}.$$

\end{lemma}

\begin{proof} 
If $L^n_i/2 = H^n_i $ then there is some $h \in \mathbb{N}$ such that $\{ H^n_i = h\} \cap \{ L^n_i/2 = h\}$ occurs.  Therefore $$ \{ L^n_i/2 = H^n_i \} = \bigcup_h \left \{ \{H^n_i= h \} \cap \{L^n_i/2 = h  \} \right \}.$$  Luckily the event $\{H^n_i = h \cap L^n_i/2 = h\}$ is disjoint from $\{ H^n_i = h' \cap L^n_i/2 = h' \}$ for $h\neq h'.$  Then
$$\prob\left(\bigcup_h  \{H^n_i= h \} \cap \{L^n_i/2 = h  \} \right )  = \sum_h \prob\left(\{H^n_i= h \} \cap \{L^n_i/2 = h  \} \right).$$
If $h \notin ( \max(0,H^n_j - (j-i), h_0 + i)$ then $\prob( H^n_i = h) = 0.$  Otherwise we have 
$$\prob\left( \{ H^n_i = h \} \cap \{ L^n_i/2 = h\} \right) = \prob\left( L^n_i/2 = h | H^n_i = h\right) \prob(H^n_i = h ).$$
Combining Lemmas \ref{probxih} and \ref{problihicond} provides the result.  
\end{proof}

\begin{lemma} \label{problihicondapprox}

Let $0<i<j\leq O(n^{0.9})$ with $i>2n^{0.6}$ and $j-i > 2n^{0.6},$ and let $h_0 \in (n^{0.499},n^{0.501})$.  Define $m$ and $m_j$ such that $h_j = h_0 + m_j $ and $h = h_0 + m$  where  $|m_j|< \min(j^{0.6},n^{0.451}).$  Let $m_{max} = \min ( i^{0.6}, m_j+(j-i)^{0.6}$ and $m_{min} = \max( -i^{0.6}, m_j - (j-i)^{0.6}).$  For $m_{min}<m<m_{max}$

$$\prob( L^n_i/2=h | H^n_i = h) \prob(H^n_i = h ) =  \frac{ \sqrt{j}}{4\pi \sqrt{i(j-i)(h_0)^3} }  \exp\left(- \frac{ ( jm - im_j )^2}{4ij(j-i)}\right)(1+\Delta )   $$
where $\Delta =o(n^{-0.001} )$ uniformly in $i,j,m,m_j,h_0$ that satisfy the above conditions.  

For $m <m_{min}$ or $m>m_{max},$

$$\prob(L^n_i/2 = h|H^n_i =h)\prob(H^n_i = h ) < \exp(-n^{0.001}).$$  
\end{lemma}

\begin{proof} 

The summand in Lemma \ref{problihi} is given by 

$$\prob(L^n_i/2 = h|H^n_i=h) \prob( H^n_i = h ) = C_{h-1} \left|\cE_{v_0,h_0}^{2i-m,h} \right | \left| \cE_{2i-m+2h-1,h-1}^{v_j,h_j}\right| \left| \cE_{v_0,h_0}^{v_j,h_j}\right|^{-1}.$$

By Lemma \ref{pathcountapprox} we can make the following substitutions:

\begin{align*}
C_{h-1} &= \frac{4^{h-1}}{\pi^{1/2}h^{3/2}} (1+\Delta_1(h))),\\
 \cE_{v_0,h_0}^{2i-m,h} &= { 2(i-1) - (m-1) \choose i-1 } = \frac{ 4^i}{2^{m+1}\pi^{1/2}i^{1/2}}e^{-m^2/4i}(1+\Delta_2(i,m)),\\
 \cE_{2i-m+2h-1,h-1}^{v_j,h_j} &= { 2(j-i-h) - (m_j-m) \choose j-i-h }\\
 & = \frac{ 4^{j-i-h}}{2^{m_j-m}\pi^{1/2}(j-i-h)^{1/2}}e^{-(m_j-m)^2/4(j-i-h)}(1+\Delta_2(j-i-h,mj-m)),\\
 & = \frac{ 4^{j-i-h}}{2^{m_j-m}\pi^{1/2}(j-i)^{1/2}}e^{-(m_j-m)^2/4(j-i)}(1+\Delta_2(j-i-h,mj-m)),\\
 \cE_{v_0,h_0}^{v_j,h_j} &= { 2(j-1) - (m_j-1) \choose j-1 } = \frac{ 4^j}{2^{m_j+1}\pi^{1/2}j^{1/2}}e^{-m_j^2/4j}(1+\Delta_2(j,m_j))\\
 \end{align*} where both $\Delta_1$ and $\Delta_2$ are bounded uniformly by $n^{-0.01}$ over all parameters satisfying the conditions of the lemmas.  Combining these equations together proves the first statement of Lemma \ref{problihicondapprox}.  For the second statement we use the second approximation in Lemma \ref{theelemma} to bound $\prob(H^n_i= h_0 + m)$ using the formula in Lemma \ref{probxih}.
\end{proof}

Let's consider the special case where $j\approx n^{0.9}.$

\begin{lemma} \label{problihiapprox}

For $X^n$ chosen uniformly from $\cE_{v_0,h_0}^{v_j,h_j}$ with $i,j,h_0,$ and $h_j$ satisfying

\begin{itemize}
\item $j = n^{0.9}(1+\Delta'), \Delta'\leq n^{-0.1}$ uniformly.   
\item $n^{0.499}<h_0 < n^{0.501}.$
\item $i  \in ( 2n^{0.6}, n^{0.9}-2n^{0.6}),$
\item $h_j = h_0 + m_j$ where $|m_j| \leq n^{0.451},$
\end{itemize}
then

$$\prob( L^n_i/2 = H^n_i-1 ) = \frac{ 1 }{2 \pi^{1/2} (h_0)^{3/2} } (1+\Delta),$$ where $\Delta = \Delta(i,h_0,h_j) = o(n^{-0.001})$ uniformly in $i,h_0,h_j$ in the ranges above.

\end{lemma}

\begin{proof}

Let $m_{min}  = \max ( -i^{0.6}, m_j -(j-i)^{0.6} )$ and $m_{max} = \min( i^{0.6}, m_j+(j-i)^{0.6})$ and consider the inequality which follows from Lemma \ref{problihi}.  

\beqlbl \label{problihisum}
\sum_{m=m_{min}}^{m_{max}}\prob(L^n_i/2=h_0 + m| H^n_i=h_0+m)\prob(H^n_i=h_0 + m) 
\eeqlbl
$$\leq \prob(L^n_i/2 =H^n_i )$$
$$ \leq ne^{-n^{0.001}} + \sum_{m=m_{min}}^{m_{max} }\prob(L^n_i/2=h_0 +m| H^n_i=h_0 + m)\prob(H^n_i=h_0 +m).$$

Lemma \ref{problihicondapprox} gives 

$$\prob(L^n_i/2 = H^n_i) $$
$$= (1+o(n^{-0.01}))\sum_{m=m_{min}}^{m_{max}}  \frac{ \sqrt{j}}{4\pi \sqrt{i(j-i)(h_0)^3} }  \exp\left(- \frac{ ( jm - im_j )^2}{4ij(j-i)}\right)(1+o(n^{-0.01}) ).$$  
$$ = (1+o(n^{-0.001}))\int_{m_{min}}^{m_{max}} \frac{ 1}{4\pi \sqrt{i(1-i/j)(h_0)^3} }  \exp\left(- \frac{ ( m - im_j/j )^2}{4i(1-i/j)}\right)(1+o(n^{-0.01}) )dm.$$

By our definition $(m_{min} - \frac{i}{j}m_j ) < -n^{0.01}$ and $(m_{max} - \frac{i}{j}m_j ) > n^{0.01}$.   Therefore the integral above is computed in the standard way, with

$$
\int_{-t}^t \exp\left( -\frac{ (m- m_0)^2 }{4c}\right) dm  = 2c^{1/2}\pi^{1/2} + \delta(t).
$$ where $\delta(t)$ is an error function with exponential decay.   

$$\prob(L^n_i/2 =H^n_i ) = \frac{1}{2\pi^{1/2}(h_0)^{3/2}}(1+o(n^{-0.001})).$$

\end{proof}

\begin{corollary} \label{problihiapprox2}

For any $k \in \R$ and $\alpha \in (0,.48)$ let $\Gamma^n$ chosen uniformly from $\cE_{v_0,h_0}^{v_j,h_j}$ with $i,j,h_0,$ and $h_j$ satisfying

\begin{itemize}
\item $j = n^{0.9}(1+ \Delta'), \Delta' < n^{-0.1}$ uniformly. 
\item $n^{0.499}<h_0 < n^{0.501}.$
\item $i  \in ( 2n^{0.6}, n^{0.9}-2n^{0.6}),$
\item $h_j = h_0 + m_j$ where $|m_j| \leq n^{0.451},$
\end{itemize}
then

$$\prob( L^n_i/2 = H^n_i-k(i(n-i)/n)^\alpha ) = \frac{ 1 }{2 \pi^{1/2} (h_0)^{3/2} } (1+\Delta),$$ where $\Delta = \Delta(i,h_0,h_j) = o(n^{-0.001})$ uniformly in $i,h_0,h_j$ in the ranges above.

\end{corollary}
\begin{proof}
The proof goes exactly as in Lemma \ref{problihiapprox} with $L^n_i =H^n_i$ replaced by $L^n_i =H^n_i - k(i(n-i)/n)^\alpha$.  The order of $k(i(n-i)/n)^\alpha$ is less than $n^{0.49}$ so it will not affect the approximation.
\end{proof}

\begin{lemma} \label{isinwindow}

Fix $0<a<b<1$ and let $a_k = \lfloor an + nk/K\rfloor$ where $K = \lfloor (b-a) n^{0.1} \rfloor.$  For $\Gamma^n \in \dyck{2n}$ chosen uniformly at random, 

$$\prob\left( \bigcap_{k=0}^{K}\{n^{0.499} < \Gamma^n({V^n_{a_k}})< n^{0.501}\} \right) > 1 -  o(1).$$

\end{lemma}

\begin{proof}

This is an immediate consequence of 
Corollary \ref{corollary V-shift} along with the convergence of Dyck paths to Brownian excursion.  
\end{proof}

\begin{lemma} \label{lowdeviation}

Fix $0<a<b<1.$  For any $n$ large enough and $\Gamma^n \in \dyck{2n}$,

$$\prob\left( \bigcup_ {k=0}^{K-1}\left \{ |\Gamma^n({V^n_{a_k}}) - \Gamma^n( V^n_{a_{k+1}} )| > n^{0.451} \right\}  \Big | \bigcap_k \left\{ n^{0.499} < \Gamma^n({V^n_{a_k}} )< n^{0.501}\right\}\right) < e^{-n^{0.0001}}.$$

\end{lemma}

\begin{proof}

This follows by a similar argument to Corollaries \ref{corollary V-shift} and \ref{corollary p1}, with slight modifications to the parameters.
\end{proof}

\begin{lemma} \label{allhigh}

For sufficiently large $n$, for every $\frac{1}{2}n^{0.9} < j\leq 2n^{0.9},$ and $h_0,h_j$ both bounded between $n^{0.499}$ and $n^{0.501}$ with $|h_0-h_j| < n^{0.451}$ we have that if $X^n \in \cE_{v_0,h_0}^{v_{j},h_j}$ is chosen uniformly at random then

$$\prob\left( \sup_{t\in [0,1]} \left |X^n( v_0 + tv_j ) - h_0\right| > 2 n^{0.452} \right) < e^{-n^{0.001}}.$$

\end{lemma}

\begin{proof}

First note that the maximum fluctuation of $|X^n(v_0 + tv_j) - h_0|$ is within 1 of the maximum fluctuation of $|H^n_i- h_0|$ for $i \leq j$.  Suppose $i > j/2$, $H^n_i=h$, and  $|h -h_0| > n^{0.452}$.  Lemma \ref{pathcountapprox} gives 
$$|\cE_{v_0,h_0}^{V^n_i,h}| \leq  \frac{4^{i}}{2^{h-h_0+1}}e^{-n^{0.904}/4j},$$
for sufficiently large $n$, independent of $j,h_0$, and $h_j$ satisfying the hypotheses of the lemma.  There are $2(j-i)-(h_j-h)$ steps remaining to go from $(v_i,h)$ to $(v_j,h_j).$  So 
$$|\cE_{V^n_i,h}^{v_j,h_j}| \leq \frac{ 4^{j-i}}{2^{h_j-h}}.$$
Again by Lemma \ref{pathcountapprox}, $$\left|\cE_{v_0,h_0}^{v_j,h_j}\right| \geq  \frac{4^j}{2^{h_j-h_0+2}\sqrt{\pi j}}\exp\left( -\frac{(h_j-h_0)^2}{4j}\right),$$
for sufficiently large $n$, independent of $j,h_0$, and $h_j$ satisfying the hypotheses of the lemma.  Then we can conclude that 
$$\prob\left( H_i^n =h \right) \leq 4\sqrt{\pi j} \exp\left(-\frac{n^{0.904}}{4j} + \frac{(n^{0.451})^2}{4j}\right) \leq 8\sqrt{\pi} n^{0.9} e^{-n^{0.004}/8}.$$
A similar bound can be used for $i < j/2$ with a little more work.  Note that if $|h - h_0| > n^{0.452}$ and $|h_j-h_0|<n^{0.451}$ then $|h-h_j| > \frac{1}{2}n^{0.452}.$  The same argument now applies.  Using the union bound and summing over the possible values of $h$ and $i$ now gives the result. 
\end{proof}

\begin{lemma} \label{fixsmall}

Let $I \subset[an,bn]$ denote an interval of length at most $n^{\alpha}.$  For $\gamma \in \dyck{2n},$ if $h_i > n^{0.49}$ for all $i\in I$, then 

$$\theta_I \leq n^{\alpha - 0.49}.$$ 

\end{lemma}

\begin{proof} 
If the $j$th excursion is contained in $i$th excursion, then both $h_j > h_i$ and $l_j/2 < l_i/2.$  If $\theta_i=1$ then for $i<j< l_i/2$, $\theta_j=0.$  

For an interval $I$ with $|I|<n^{0.49}$ suppose at least one fixed point exists.  Let $i^*$ denote the first excursion that satisfies $\theta_{i*}=1.$  Since $i^*$ corresponds to a fixed point, $$l_{i^*} = h_{i^*} > n^{0.49}.$$  Therefore, for all $j\in I$ such that $j>i^*$ the $j$th excursion is contained in the $i^*$th excursion and $\theta_j=0$.  Therefore either $\theta_{I} = 0$ if no such $i^*$ exists, or $\theta_{I} = \theta_{i^*}=1.$  For an interval of size less than $n^{\alpha}$, it can be covered by $n^{\alpha-0.49}$ intervals of size $n^{0.49}$ each of which has at most one fixed point so the total number of fixed points will be bounded by $n^{\alpha-0.49}.$
\end{proof}

\begin{corollary} \ \label{fixedsmall}
Let $I \subset[an,bn]$ denote an interval of length at most $n^{\alpha}.$  For $\gamma \in \dyck{2n},$ if $h_i > n^{0.49}$ for all $i\in I$, then 

$$\theta^{K,\alpha}_I \leq 2n^{\alpha - 0.49}.$$ 

\end{corollary}

\begin{proof}
The function $f(i)=(i(n-i)/n)^{\alpha}$ can change by at most 1 over any interval of length $n^{.49}$. This implies that over any interval $I'$ of length $n^{.49}$ has $\theta^{K,\alpha}_{I'}\leq 2$. Then the result follows as in Lemma \ref{fixsmall}.
\end{proof}

\begin{lemma}\label{uglyvarcalc}

There exists constant $C>0$ such that  or every $n$ large enough and $i,j,w$ that satisfy the following:

\begin{enumerate}

\item $2n^{0.6}<i<j<n^{0.9}-2n^{0.6},$ 
\item $|i-j| > 2n^{0.6},$
\item and $w < n^{0.451}$

\end{enumerate}

$$\Psi(i,j,w,n) = \sum_{m' = -\infty}^{\infty} \sum_{ m = -\infty }^{\infty} \exp\left ( - \frac{1}{4} \left (\frac{m^2}{i} + \frac{(m-m')^2}{j-i} + \frac{(w-m')^2}{n^{0.9}-j} - \frac{w^2}{n^{0.9}} \right) \right) $$
$$< C \sqrt{ \frac{i (j-i) (n^{0.9} - j) }{n^{0.9}}}.$$

\end{lemma}

\begin{proof}

Our goal will be to convert this double sum into a recognizable form. 

\begin{align*}
\Psi(i,j,w,n) \leq& \sum_{m'=-\infty}^{\infty}\sum_{m = -\infty}^{\infty} \exp\left ( - \frac{1}{4} \left (\frac{m^2}{i} + \frac{(m-m')^2}{j-i} +\frac{(w-m')^2}{n^{0.9}-j} - \frac{w^2}{n^{0.9}}   \right) \right) \\
\leq&\sum_{m' =-\infty}^{\infty} \left[\exp\left( -\frac{1}{4} \left( \frac{(w-m')^2}{n^{0.9}-j} - \frac{w^2}{n^{0.9}} \right) \right) G( i,j,m')\right]
\end{align*}
where 
$$G( i,j,m' ) =  \sum_{m= -\infty}^{\infty}\exp \left(-\frac{1}{4} \left( \frac{m^2}{i} + \frac{(m-m')^2}{j-i} \right)\right ).$$

With some algebra we see

\begin{align*}
G(i,j,m') =& \exp\left ( -\frac{1}{4}  \frac{ {m'}^2}{j} \right) \sum_{m=-\infty}^{\infty} \exp\left ( -\frac{1}{4}  \frac{j(m-m'i/j)^2}{i(j-i)} \right )\\
\leq & C_1 \exp\left ( -\frac{1}{4}  \frac{ {m'}^2}{j} \right)  \int_{-\infty}^{\infty} \exp\left ( -\frac{1}{4}  \frac{j}{i(j-i)} (t-m'i/j)^2\right )dt\\
\leq &   C_1 \exp\left ( -\frac{1}{4}  \frac{ {m'}^2}{j} \right) \sqrt{ \frac{i(j-i)}{j} }  
\end{align*}
for some positive constant $C_1$ that does not depend on $i,j$ or $m'$.  Inserting this into the upper bound for $\Psi(i,j,w,n)$ gives

\begin{align*}
\Psi(i,j,w,n) \leq& C_1\sqrt{ \frac{i(j-i)}{j} }  \sum_{m'=-\infty}^{\infty}\exp \left( -\frac{1}{4}\left( \frac{{m'}^2}{j} +  \frac{(m'-w)^2}{n^{0.9}-j} - \frac{w^2}{n^{0.9}} \right)\right)\\
\leq&C_1\sqrt{ \frac{i(j-i)}{j} }\sum_{m' = -\infty}^{\infty} \exp\left( -\frac{1}{4}\left( \frac{n^{0.9} ( m' - wj/n^{0.9} )^2}{j(n^{0.9}-j)}\right) \right)\\
\leq &C \sqrt{ \frac{i(j-i)(n^{0.9}-j)}{n^{0.9}} }
\end{align*}
where $C>0$ and does not depend on $i,j,w,$ and $n$.
\end{proof}

The next two results are special cases of \cite[Theorem III.12, Theorem III.15]{Petrov1} respectively (see also \cite[Lemma A1, Lemma A2]{MaMo03}). 

\begin{lemma}\label{lemma maximal inequality}
Let $X_1,X_2,\dots$ be i.i.d with $\E X_1=0$ and let $S_n =X_1+\cdots + X_n$. Suppose that $\sigma^2= \E(X_1^2) < \infty$.  For all $x$ and $n$ we have
\[\P\left(\max_{1\leq k\leq n} S_k \geq x\right) \leq 2 \P\left(S_n\geq x - \sqrt{2n\sigma^2}\right).\]
\end{lemma}

\begin{lemma}\label{lemma moderate deviations}
Let $X_1,X_2,\dots$ be i.i.d with $\E X_1=0$ and let $S_n =X_1+\cdots + X_n$.  Suppose that there exists $a>0$ such that $\E(e^{t |X_1|}) < \infty$.  Then there exist constants $g, T>0$, independent of $n$, such that
\[ \P(S_n\geq x) \leq \begin{cases} \exp\left(-\frac{x^2}{2gn}\right) &  \textrm{if } 0\leq x\leq ngT \\ \\ \exp\left(-\frac{Tx}{2}\right) & \textrm{if }  x\geq ngT.\end{cases}\]
\end{lemma}

These lemmas lead immediately to the following corollary. 

\begin{corollary}\label{corollary fluctuations}
Maintaining the hypotheses of Lemma \ref{lemma moderate deviations}, fix $\epsilon, c>0$ and $0< \alpha< 2\beta$ and let $\nu  = \min(\beta, 2\beta-\alpha)$.  There exist constants $A,B >0$ such that
\[ \P\left( \max_{1\leq i\leq n} \max_{|i-j|\leq cn^{\alpha}} |S_j -S_i| \geq \epsilon n^\beta \right)\leq A \exp\left( -B n^{\nu}\right)\]
\end{corollary}

Let $S=(S_m,m\geq 0)$ be a simple symmetric random walk on $\Z$ with $S_0=0$.  Define $V_0=0$ and for $m\geq 1$ let $V_m = \inf\{k>V_{m-1} : S_k-S_{k-1} = 1\}$.  Let $\eta(S) = \inf\{ k : S_k=-1\}$.  Observe that $(V_m-V_{m-1})_{m\geq 1}$ is an i.i.d sequence of geometric random variables with parameter $1/2$.

\begin{corollary}\label{corollary V-shift}
Fix $\epsilon >0$ and $1/2 < \alpha\leq 1$.  There exist constants $A, B>0$ such that 
\[ \P\left( \max_{1\leq i\leq n} |V_i - 2i| \geq \epsilon n^\alpha\right) \leq A \exp\left(-B n^{2 \alpha -1}\right).\]
\end{corollary}

\begin{corollary} \label{corollary p1}
Let $\Gamma^n$ be a uniformly random Dyck path of length $2n$.  For any $\delta>0$, there exist constants $A,B,\nu>0$ such that for all $n\geq 1$
\[ \P\left(\max_{1\leq i\leq 2n} \Gamma^n_i \geq 0.4 n^{0.5 +\delta}\right) \leq A\exp\left(-Bn^\nu\right) \]
and 
\[ \P\left( \max_{1\leq i\leq 2n} \max_{|i-j|\leq 2n^{0.5+\delta}} |\Gamma^n_j -\Gamma^n_i| \geq  0.5 n^{0.25 + \delta} \right)\leq A \exp\left( -B n^{\nu}\right)\]
\end{corollary}

\begin{proof}
Noting that 
\[ \Gamma^n \overset{d}{=} (S_k, 0\leq k \leq 2n) \textrm{ given } \eta(S) =2n+1,\]
the first claim is an immediate consequence of Lemmas \ref{lemma maximal inequality} and \ref{lemma moderate deviations} combined with the fact that $\P(\eta(S) =2n+1) \sim cn^{-3/2}$ for some $c>0$.  The second claim follows similarly from Corollary \ref{corollary fluctuations}
\end{proof}

\section*{Acknowledgement}
We would like to thank Lerna Pehlivan for many helpful suggestions.  We would also like to thank Igor Pak for a helpful conversation.

\bibliographystyle{plain}

\end{document}